\documentclass[reqno]{amsart}
\usepackage{amsmath, amsthm, amssymb, amstext}
\usepackage[left=2.5cm,right=2.5cm,top=2.5cm,bottom=2.5cm]{geometry}
\usepackage{hyperref,xcolor}
\hypersetup{pdfborder={0 0 0},colorlinks}
\usepackage{enumitem}
\usepackage{cleveref}
\allowdisplaybreaks
\newtheorem{theorem}{Theorem}
\newtheorem{remark}[theorem]{Remark}
\newtheorem{lemma}[theorem]{Lemma}
\newtheorem{proposition}[theorem]{Proposition}
\newtheorem{corollary}[theorem]{Corollary}
\newtheorem{definition}[theorem]{Definition}

\numberwithin{theorem}{section}
\numberwithin{equation}{section}

\theoremstyle{plain}
\theoremstyle{definition}
\usepackage{mathtools}

\title[On Polyharmonic Double Phase Problems]{On doubly critical polyharmonic double phase problems: Existence and non-existence of solutions}
\author[A. Dixit]{Ashutosh Dixit}
\address[A. Dixit]{Department of Mathematics, Indian Institute of Technology Jodhpur, Rajasthan 342030, India}
\email{p23ma0014@iitj.ac.in}

\author[T. Mukherjee]{Tuhina Mukherjee$^*$}
\address[T. Mukherjee]{Department of Mathematics, Indian Institute of Technology Jodhpur, Rajasthan 342030, India}
\email{tuhina@iitj.ac.in}
\thanks{$^*$corresponding author}


\begin{document}
\begin{abstract}
In this article, we investigate the existence and nonexistence of weak solutions to higher-order doubly critical elliptic problems with weights, driven by a polyharmonic double phase operator. More precisely, we deal with the following problem
\begin{equation*}
\begin{cases}
\mathcal{L}^m_{p,q}(u) = f(x,u) ~&\text{in } \Omega,\\[6pt]
    u=\nabla u=\cdots\nabla^{m-1} u=0
    &\text{on  }{\partial\Omega},
\end{cases}
\end{equation*}
where $\Omega \subset \mathbb{R}^N$ with $N \geq 2$ is a smooth bounded domain with Lipschitz boundary $\partial\Omega$, $m \in \mathbb{N}$, $1 < p < q < \frac{N}{m}$ with $(N-1)q\leq Np$, the nonlinear term $f\colon\Omega\times\mathbb{R}\to \mathbb{R}$ is a Carath\'{e}odory function, which has doubly critical growth, and $\mathcal{L}^m_{p,q}$ represents a polyharmonic double phase operator. 
By establishing new compactness results within a suitable Musielak–Orlicz–Sobolev framework and applying variational methods, we prove the existence of nontrivial weak solutions. In addition, we derive nonexistence results under appropriate assumptions by establishing a Pohozaev-type identity for higher-order derivatives. Our approach extends classical techniques to capture the intricate features of the double-phase operator for higher-order derivatives, and addresses the difficulties arising from critical nonlinearities, in particular extending the results of [F. Colasuonno, K. Perera, J. Differ. Equ., 422 (2025), 426–488] in a polyharmonic double phase setup overcoming the non-closedness of truncations in higher-order Sobolev spaces.
\end{abstract}

\maketitle 

\section{Introduction}
Nonlinear partial differential equations play a central role in contemporary mathematical analysis, arising naturally in the modeling of diverse phenomena in physics, engineering, and materials science. In particular, higher-order elliptic equations appear in several important contexts, including the theory of thin elastic plates, stationary surface diffusion flows, the Paneitz--Branson equation, and the Willmore equation, also known in biophysics as the Helfrich model; see \cite{Gazzola-Grunau-Sweers-2010}. A classical example is the biharmonic equation describing the bending of thin elastic plates within the Kirchhoff--Love theory \cite{Love-1944}. Moreover, higher-order terms arise naturally in strain-gradient elasticity and continua with microstructure, where they account for internal stresses and microscopic deformations \cite{Antman-2005,Mindlin-1964}. Additional applications occur in nonlinear optics \cite{Fibich-Ilan-Papanicolaou-2002} and in phase separation models such as the Cahn--Hilliard equation \cite{Cahn-Hilliard-1958}.

From the analytical perspective, the passage from second-order to higher-order elliptic equations introduces substantial difficulties. Many tools available for the Laplacian $\Delta$ do not extend to the biharmonic operator $\Delta^2$ or its higher-order counterparts. In particular, higher-order operators lack a maximum principle and do not satisfy P\'olya--Szeg\H{o}-type inequalities. Regularity theory is also more delicate, since even $(\Delta u)^2\in L^1(\mathbb{R}^N)$ does not guarantee $\Delta|u|\in L^1_{\mathrm{loc}}(\mathbb{R}^N)$. Furthermore, the Green's function associated with higher-order operators may change sign even in simple domains, complicating positivity arguments and representation formulas. These features necessitate analytical tools specifically adapted to higher-order problems.

A milestone in the study of nonlinear higher-order elliptic equations can be traced back to the work of Kratochv\'{\i}l and Ne\v{c}as \cite{Kratohvil-Necas-1971}, followed by Dr\'abek and \^Otani \cite{Drabek-Otani-2001}, who investigated problems involving the fourth-order nonlinear operator
\[
\Delta_p^2 u=\Delta\big(|\Delta u|^{p-2}\Delta u\big),
\]
known as the $p$-biharmonic operator. For $p=2$, this operator reduces to the classical biharmonic operator $\Delta^2$, which also appears in fluid mechanics as a viscosity-related term. More generally, for $m\ge2$, the polyharmonic operator is defined by
\[
\Delta^m u=\underbrace{\Delta(\Delta\cdots\Delta}_{m\ \text{times}}u),
\]
providing a natural framework for modeling complex mechanical and structural phenomena \cite{Gazzola-Grunau-Sweers-2010}. A substantial literature is available on elliptic problems driven by polyharmonic operators with critical nonlinearities, as well as on elliptic equations exhibiting nonstandard growth. For general results on polyharmonic equations, we refer to Bernis--Grunau \cite{Bernis-Grunau-1995} and Ge--Wei--Zhou \cite{Ge-Wei-Zhou-2011}, while the analysis of nonstandard growth problems is typically carried out within the framework of Musielak--Orlicz spaces.

In the polyharmonic setting, Gazzola \cite{Gazzola-1998} investigated elliptic problems with critical growth and emphasized the bifurcation phenomenon of critical dimensions previously observed by Pucci-Serrin \cite{Pucci-Serrin-1990}. Subsequently, Clapp-Squassina \cite{Clapp-Squassina-2003} established multiplicity results for perturbed polyharmonic problems with critical growth under symmetry assumptions, while Bartsch-Schneider-Weth \cite{Bartsch-Schneider-Weth-2004} proved the existence of sequences of nodal finite-energy solutions. Further developments include the analysis of bubbling solutions by Guo-Peng-Yan \cite{Guo-Peng-Yan-2015}, as well as recent contributions by Cannone-Cingolani-Mederski \cite{Cannone-Cingolani-Mederski-2025}.

In recent years, significant attention has been devoted to problems involving nonstandard growth conditions, particularly those driven by double phase operators. Such operators arise from variational integrals of the form
\begin{equation}
\omega \longmapsto \int_\Omega \left(|\nabla\omega|^p+a(x)|\nabla\omega|^q\right)\,dx,
\label{eq:Zhikov}
\end{equation}
introduced by Zhikov \cite{Zhikov-1986,Zhikov-1995,Zhikov-1997} to model strongly anisotropic materials. The presence of the modulating coefficient $a(x)\ge0$ allows the energy density to switch between distinct growth regimes, leading to a rich and highly nontrivial analytical structure. The regularity theory and variational analysis of double phase problems have since been extensively developed; see, for instance, \cite{Baroni-Colombo-Mingione-2015,Baroni-Kuusi-Mingione-2014,Colombo-Mingione-2015,Baroni-Colombo-Mingione-2018}.

On the other hand, double phase problems with critical nonlinearities have attracted increasing attention in recent years. Papageorgiou-Zhang \cite{Papageorgiou-Zhang-2020} studied parametric double phase equations involving critical terms and proved the existence of multiple nontrivial solutions. Liu-Papageorgiou \cite{Liu-Papageorgiou-2021} analyzed Dirichlet double phase problems with unbalanced growth. More recently, Farkas-Fiscella-Ho-Winkert \cite{Farkas-Fiscella-Ho-Winkert-2025} examined double phase equations with mixed (sandwich-type) critical growth and established existence and multiplicity results. In addition, Colasuonno-Perera \cite{Colasuonno-Perera-2025} revisited a Brezis--Nirenberg-type double phase problem in the first-order case $m=1$, proving compactness, existence, and nonexistence results.

While polyharmonic and double phase problems have each been studied extensively, their combination, namely \emph{polyharmonic double phase problems}, has not been systematically investigated. This gap in the literature provides the primary motivation for the present work. We extend the variational framework and compactness theory developed in \cite{Colasuonno-Perera-2025} to a genuinely higher-order setting, where new analytical difficulties arise from the interaction of polyharmonic operators, nonstandard growth, and critical nonlinearities. Although our work is inspired by the variational strategy developed in \cite{Colasuonno-Perera-2025}, the extension to the polyharmonic setting is far from formal. Several fundamental tools available in the first-order case break down for $m\ge2$. In particular, truncation and positivity arguments are no longer available, higher-order derivatives do not satisfy chain rules compatible with the Musielak--Orlicz structure, and weak convergence does not yield convergence of the associated nonlinear terms. As a result, new compactness arguments are required, especially in the proof
of Lemma~\ref{lemmaweak-convergence}, where the almost everywhere convergence of higher-order gradients
is obtained by handling the interaction between higher-order derivatives and double
phase growth. These aspects do not arise in the first-order setting and constitute a substantial part of the novelty of the present work.

We consider the following boundary value problem:
\begin{equation}\label{3.1.1}\tag{$\mathcal{P}_f$}
\begin{cases}
\mathcal{L}^m_{p,q}(u) = f(x,u) & \text{in } \Omega, \\[6pt]
u = \nabla u = \cdots = \nabla^{m-1} u = 0 & \text{on } \partial\Omega,
\end{cases}
\end{equation}
where $\mathcal{L}^m_{p,q}$ denotes a polyharmonic double phase operator defined by
\begin{equation}
\mathcal{L}^m_{p,q}(u)=
\begin{cases}
-\nabla\cdot\Big(\Delta^{\frac{m-1}{2}}\big[|\nabla\Delta^{\frac{m-1}{2}}u|^{p-2}\nabla\Delta^{\frac{m-1}{2}}u
+a(x)|\nabla\Delta^{\frac{m-1}{2}}u|^{q-2}\nabla\Delta^{\frac{m-1}{2}}u\big]\Big),
& \text{if } m \text{ is odd}, \\[2ex]
\Delta^{\frac{m}{2}}\Big(|\Delta^{\frac{m}{2}}u|^{p-2}\Delta^{\frac{m}{2}}u
+a(x)|\Delta^{\frac{m}{2}}u|^{q-2}\Delta^{\frac{m}{2}}u\Big),
& \text{if } m \text{ is even}.
\end{cases}
\end{equation}
Here $m\in\mathbb{N}$, $\Omega\subset\mathbb{R}^N$ $(N\ge2)$ is a bounded domain, and
$a(\cdot)\in C^{0,\alpha}(\mathbb{R}^N)$ for some $\alpha\in(0,1]$, with $a(x)\ge0$.
The nonlinearity $f:\Omega\times\mathbb{R}\to\mathbb{R}$ exhibits doubly critical growth and is given by
\[
f(x,u)=\mu|u|^{p^*-2}u+b(x)|u|^{q^*-2}u+c(x)|u|^{s-2}u,
\]
where $1<p<q<\frac{N}{m}$ with $(N-1)q\le Np$, $p^*=\frac{Np}{N-mp}$,
$q^*=\frac{Nq}{N-mq}$, $1<s<q^*$, and $\mu\ge0$.
The functions $b\in C(\overline{\Omega})$ and $c\in L^\infty(\Omega)$ are nonnegative and satisfy
\begin{equation}\label{3.1.2}\tag{A0}
a_0=\inf_{x\in\operatorname{supp}(b)} a(x)>0,
\end{equation}
and
\begin{equation}\label{3.1.3}\tag{A1}
c(x)\le C\,a(x)^{\frac{s}{q}}\quad \text{for all } x\in\Omega,
\end{equation}
for some constant $C>0$ whenever $s\ge p^*$.

For $j\in\mathbb{N}$, the notation $\nabla^j u$ denotes the $j$-th order gradient of $u$, defined by
\[
\nabla^j u=
\begin{cases}
\nabla\Delta^{\frac{j-1}{2}}u, & \text{if } j \text{ is odd},\\[4pt]
\Delta^{\frac{j}{2}}u, & \text{if } j \text{ is even},
\end{cases}
\]
where $\nabla$ and $\Delta$ denote the standard gradient and Laplacian operators, respectively.

Assumption \eqref{3.1.2} plays a crucial role in our analysis, as it ensures that the coefficient $a(x)$ is strictly positive on the support of $b(x)$. This allows us to localize the critical term involving $|u|^{q^*-2}u$ in regions where the $q$-growth dominates, thereby enabling the use of classical Sobolev embeddings and sharp modular inequalities. In particular, condition \eqref{3.1.2} is essential for deriving the key estimate \eqref{eq3.3.8}, which guarantees that the weak limit of a $(PS)_\beta$ sequence associated with \eqref{3.1.1} is nontrivial (see Proposition~\ref{propcompactness}).
On the other hand, the term $|u|^{s-2}u$ acts as a subcritical perturbation of the critical components of the nonlinearity. Assumption \eqref{3.1.3} ensures that this perturbation remains genuinely subcritical, even in the borderline case $s\ge p^*$, and is instrumental in the compactness analysis (cf.\ Proposition~\ref{prop3.3.12}).
By combining variational methods with the mountain pass theorem, we establish new existence and nonexistence results for weak solutions of problem \eqref{3.1.1}. Our approach extends classical variational frameworks to accommodate the polyharmonic double phase structure and overcomes the compactness difficulties induced by critical nonlinearities.

We now summarize the main features and novelties of the present work.
\begin{itemize}
\item[(a)] The leading operator in \eqref{3.1.1} is a \emph{polyharmonic double phase operator}, whose ellipticity may switch between different regimes due to the presence of a modulating coefficient $a(\cdot)\in C^{0,\alpha}(\Omega)$.

\item[(b)] The presence of doubly critical nonlinearities in \eqref{3.1.1} leads to serious compactness difficulties in the underlying Musielak--Orlicz--Sobolev space (see Definition~\ref{definitionsobolev}), making it highly nontrivial to prove that weak limits of minimizing sequences are solutions of the problem.

\item[(c)] Since the space $W_0^{m,\mathcal{H}}(\Omega)$ is not a Hilbert space, weak convergence of a bounded Palais--Smale sequence $(u_j)_{j\in\mathbb{N}}$,
\[
u_j \rightharpoonup u \quad \text{in } W_0^{m,\mathcal{H}}(\Omega),
\]
does not imply the weak convergence of the corresponding nonlinear terms, namely
\[
|\nabla^m u_j|^{p-2}\nabla^m u_j \rightharpoonup |\nabla^m u|^{p-2}\nabla^m u 
\quad \text{in } L^{\frac{p}{p-1}}(\Omega).
\]
Therefore, a refined analysis is required to establish the almost everywhere convergence $\nabla^m u_j\to \nabla^m u$ in $\Omega$.

\item[(d)] In general, for $u\in W_0^{m,\mathcal{H}}(\Omega)$, neither $|u|$ nor the truncations $u^{\pm}=\max\{\pm u,0\}$ necessarily belong to $W_0^{m,\mathcal{H}}(\Omega)$. As a consequence, standard arguments for proving the existence of positive solutions cannot be applied in this framework.

\item[(e)] A crucial ingredient of our analysis is a higher-order Poho\v{z}aev-type identity, stated in Theorem~\ref{thm3.3.1}, which plays a fundamental role in deriving nonexistence results.

\item[(f)] The proofs rely on refined variational and topological methods specifically adapted to the higher-order double phase setting.
\end{itemize}

The paper is organized as follows. In Section~2, we state the main results and introduce the precise functional framework. Section~3 collects preliminary material on Musielak--Orlicz spaces and develops the variational setting for double phase problems. In Section~4, we establish compactness results and derive sharp estimates for the associated energy thresholds. Section~5 is devoted to the proof of the main existence results. Finally, Section~6 contains the proof of the higher-order Poho\v{z}aev identity and the corresponding nonexistence result.
\section{Main Results} \label{secresults}
This article establishes existence results for double-phase problems through several variational and topological tools. We investigate the following type of doubly critical growth equation driven by a polyharmonic double phase operator 
\begin{equation}\label{3.2.8}\tag{$\mathcal{P}_g$}
\begin{cases}
     \mathcal{L}^m_{p,q}(u)  = \mu|u|^{p^{*}-2}u + b(x)|u|^{q^{*}-2}u + g(x,u) & \text{in } \Omega, \\
   u=\nabla u=\cdots\nabla^{m-1} u=0
    &\text{on  }{\partial\Omega},
\end{cases}
\end{equation}
where \(g :\Omega \times \mathbb{R} \to \mathbb{R}\) is a Carath\'{e}odory function satisfying the subcritical growth as follows
\begin{equation}\label{3.2.9}
    |g(x, t)| \leq c_{1} + c_{2}|t|^{r-1} + c(x)|t|^{s-1} \quad \text{for a.a. } x \in \Omega~~\text{and for all } t \in \mathbb{R},
\end{equation}
where \(c_{1}, c_{2} > 0\) are constants,  \(1 < r < p^{*} \leq s < q^{*}\), and coefficients \(b (x)\in C(\overline{\Omega})\), and \(c(x) \in L^{\infty}(\Omega)\) satisfy \eqref{3.1.2} and \eqref{3.1.3}, respectively. We refer to Section \ref{prelim} for necessary definitions and preliminaries.
\begin{definition}
    A function $u \in W_{0}^{m,\mathcal{H}}(\Omega)$ is a weak solution of \eqref{3.2.8}, if there holds
\begin{equation}\label{weaksolution}
\int_{\Omega} \left(|\nabla^m u|^{p-2} + a(x)|\nabla^m u|^{q-2}\right)\nabla^m u \cdot \nabla^m v \, \mathrm{d}x = \int_{\Omega} \left(\mu|u|^{p^*-2}u + b(x)|u|^{q^*-2}u+g(x,u)\right)v \, \mathrm{d}x
\end{equation}
for all $v \in W_{0}^{m,\mathcal{H}}(\Omega)$, where the space $ W_{0}^{m,\mathcal{H}}(\Omega)$ is given in Definition  \ref{definitionsobolev}.
\end{definition}

Define the functional \(E :W_{0}^{m, \mathcal{H}}(\Omega) \to \mathbb{R}\)  associated with \eqref{3.2.8}  by
\begin{equation}\label{eq3.2.10}
E(u) = \int_{\Omega} \left( \frac{1}{p}|\nabla^m u|^{p} + \frac{a(x)}{q}|\nabla^m u|^{q} - \frac{\mu}{p^{*}}|u|^{p^{*}} - \frac{b(x)}{q^{*}}|u|^{q^{*}} - G(x,u) \right) \mathrm{d}x,\forall~u\in  W_{0}^{m,\mathcal{H}}(\Omega),
\end{equation}
where \(G(x, t) := \int_{0}^{t} g(x, \tau) d\tau\). It is easy to see that $E$ is well-defined, of class $C^1$ on $ W_{0}^{m,\mathcal{H}}(\Omega)$ and its Fr\'echet and G\^ateaux differentials exist and equal such that
$$\langle{E^\prime(u),v}\rangle=\int_{\Omega}\Big( \left(|\nabla^m u|^{p-2} + a(x)|\nabla^m u|^{q-2}\right)\nabla^m u \cdot \nabla^m v -\left(\mu|u|^{p^*-2}u + b(x)|u|^{q^*-2}u+g(x,u)\right)v \Big)\, \mathrm{d}x $$
for all $u,v\in  W_{0}^{m,\mathcal{H}}(\Omega)$, where $\langle{\cdot,\cdot}\rangle$ denotes the duality pairing between the dual $(W_{0}^{m,\mathcal{H}}(\Omega))^\ast$ and $W_{0}^{m,\mathcal{H}}(\Omega)$. It is known that all the critical points of $E$ are weak solutions to \eqref{3.2.8}.
\begin{definition}{(Palais-Smale Compactness Condition)}
 Let $X$ be a Banach space and $I\colon X\to\mathbb{R}$ be a functional of class $C^1(X,\mathbb{R})$. We say that  $I$ satisfies the Palais-Smale compactness condition at a suitable level $c\in \mathbb{R}$, if for any sequence $(u_j)_{j\in\mathbb{N}}\subset X$ such that
 \begin{equation}\label{eq6.1}
 I (u_j)\to c\quad\text{and}\quad\sup_{\|\varphi\|_{X}=1}|\langle{I^\prime(u_j),\varphi}\rangle|\to 0\quad\text{as}\quad j\to\infty  \end{equation}
has a strongly convergent subsequence in $X$. The sequence $(u_j)_{j\in\mathbb{N}}\subset X$ satisfying \eqref{eq6.1} is known as a Palais-Smale sequence at level $c\in \mathbb{R}$, which is denoted by the symbol \textnormal{(PS)$_c$}.
\end{definition}
To ensure boundedness of (PS)\(_{\beta}\)-sequences for any $\beta\in\mathbb{R}$, we assume the condition on the nonlinearity $g$ and its primitive $G$ as follows
\begin{equation}\label{3.2.11}
G(x, t) - \frac{t}{\sigma}g(x, t) \leq \mu\left(\frac{1}{\sigma} - \frac{1}{p^{*}}\right)|t|^{p^{*}} + c_{3} \quad \text{for a.a. } x \in \Omega~\text{and for all }t \in \mathbb{R},
\end{equation}
for some suitable constant \(c_3 > 0\) and \(q < \sigma < p^{*}\), which help us to demonstrate the existence of a critical energy threshold \(\beta^{*} > 0\) such that every (PS)\(_{\beta}\) sequence contains a subsequence still denoted by itself, which converges to a nontrivial solution of \eqref{3.2.8} for all $0<\beta<\beta^\ast$. This threshold can be characterized variationally through the auxiliary functional
\begin{equation}\label{eq3.2.12}
    I(X, Y, Z, W) = \frac{1}{p}X + \frac{1}{q}Y - \frac{1}{p^{*}}Z - \frac{1}{q^{*}}W \quad \text{for all } X, Y, Z, W \geq 0.
\end{equation}
\noindent
 Define the optimal Sobolev constants for higher-order derivatives by
\begin{equation}\label{eq3.2.13}
    S_{p} = \inf_{\substack{u \in \mathcal{D}^{m,p}(\mathbb{R}^N)\setminus\{0\}}} \frac{\|\nabla^m u\|_{L^p(\mathbb{R}^N)}^p}{\|u\|_{L^{p^*}(\mathbb{R}^N)}^p}\quad\text{and} \quad
S_{q} = \inf_{\substack{u \in \mathcal{D}^{m,q}(\mathbb{R}^N) \setminus\{0\}}} \frac{\|\nabla^m u\|_{L^q(\mathbb{R}^N)}^q}{\|u\|_{L^{q^*}(\mathbb{R}^N)}^q},
\end{equation}
where \(\mathcal{D}^{m,r}(\mathbb{R}^N) = \{ u \in L^{r^*}(\mathbb{R}^N) : \nabla^m u \in L^r(\mathbb{R}^N) \}\) for \(r \in \{p,q\}\).
Let \(S(\mu, \|b\|_\infty)\) be the set of nonnegative quadruples \((X, Y, Z, W) \in \mathbb{R}^4\) satisfying the following hypotheses
\begin{equation}
    I(X, Y, Z, W) > 0,~X + Y = Z + W,~Z \leq \mu S_p^{-\frac{p^*}{p}} X^{\frac{p^*}{p}},~\text{and}~ W \leq \|b\|_\infty (a_0 S_q)^{-\frac{q^*}{q}} Y^{\frac{q^*}{q}}. 
    \label{eq3.2.14}
  \end{equation}
We again define a critical energy threshold by
\[
\beta^{*}(\mu, \|b\|_\infty) = \inf_{(X,Y,Z,W) \in S(\mu,\|b\|_\infty)} I(X,Y,Z,W).
\]
Now, we can state the main results of our article as listed below. 
\begin{proposition}\label{propcompactness}
Let the hypotheses \eqref{3.1.2}, \eqref{3.1.3}, \eqref{3.2.9}, and \eqref{3.2.11} be satisfied. Further, if there holds $0 < \beta < \beta^{*}(\mu, \|b\|_\infty)$, 
then problem \eqref{3.2.8} has at least one nontrivial solution.
\end{proposition}
The proofs of our main results depend upon Proposition \ref{propcompactness} combined with the asymptotic behaviour described in below.
\begin{proposition}\label{propasymptotics}
The following results hold
\begin{itemize}
    \item[\normalfont{(a)}] We have \begin{equation}\label{eq3.2.18}
    \beta^{*}(\mu, \|b\|_\infty) \geq \frac{m}{N} \frac{S_p^{N/mp}}{\mu^{\frac{N-mp}{mp}}} + o(1) \quad \text{as }\quad \|b\|_\infty \to 0^+,
    \end{equation}
    where \(\mu > 0\) is a fixed constant and the sequence $o(1)\to 0$ as $\|b\|_\infty \to 0^+$.
    \item[\normalfont{(b)}] We have
\begin{equation}\label{eq3.2.19}
    \beta^{*}(\mu, \|b\|_\infty) \geq \frac{m}{N} \frac{(a_0 S_q)^{\frac{N}{mq}}}{\|b\|_\infty^{(N-m q)/mq}} + o(1) \quad \text{as }\quad \mu \to 0^+,
     \end{equation}
     where \(\|b\|_\infty > 0\) is fixed and the sequence $o(1)\to 0$ as $\mu \to 0^+$.
\end{itemize}
\end{proposition}
\subsection{Existence of solutions to the modified problem}
Let us consider the following boundary value problem
\begin{equation}\label{3.2.1}\tag{$\mathcal{P}_\lambda$}
\begin{cases}
\mathcal{L}^m_{p,q}(u) =\lambda|u|^{r-2}u + \mu|u|^{p^*-2}u + b(x)|u|^{q^*-2}u & \text{in } \Omega, \\
u=\nabla u=\cdots\nabla^{m-1} u=0
    &\text{on  }{\partial\Omega},
\end{cases}
\end{equation}
where $p \leq r < p^*$ and $\lambda, \mu > 0$, with $b(x) \in C(\overline{\Omega})$ satisfying the hypothesis \eqref{3.1.2}.
Let $\lambda_1(p)$ be the first Dirichlet eigenvalue of  $m$-\text{th} order polyharmonic operator and given by
\begin{equation}\label{eq3.2.2star}
\lambda_1(p)=\inf_{u\in W^{m,p}_0(\Omega)\setminus\{0\}}\frac{\int_{\Omega}|\nabla^m u|^p\mathrm{d}x}{\int_{\Omega}|u|^p\mathrm{d}x},
\end{equation}
Next, we have the following theorem.
\begin{theorem}[\bf The Existence Result-I]\label{thm3.2.1}
Assume that hypothesis \eqref{3.1.2} holds and there exists a ball $B_\sigma(x_0)\subset\Omega$ such that
\begin{equation}\label{3.2.2}
a(x)=0, ~ \forall~ x\in B_\sigma(x_0).\tag{A2}
\end{equation}
Then there exists $b^*>0$ such that problem \eqref{3.2.1} possesses at least one nontrivial weak solution in $W_0^{m,\mathcal{H}}(\Omega)$, when $\mu>0$ and $\|b\|_{\infty} < b^*$, under the following hypotheses
\begin{enumerate}
    \item[\normalfont{(a)}]  $N \geq mp^2$, $r=p$, and $0 < \lambda < \lambda_1(p)$,
    \item[\normalfont{(b)}] $N \geq mp^2$, $p < r < p^*$, and $\lambda > 0$,
    \item[\normalfont{(c)}] $N < mp^2$ with $\frac{(Np - 2N + mp)p}{(N - mp)(p - 1)}<r<p^{*}$, and $\lambda > 0$.
\end{enumerate}
\end{theorem}
\begin{remark}
    For $b(x) \equiv 0$ and $\mu=1$, Theorem \ref{thm3.2.1} reduces to a Brezis-Nirenberg-type result for polyharmonic double phase operator, which is a new result of an independent interest.
Precisely, under the assumptions of \eqref{3.1.2} and \eqref{3.2.2}, the simplified problem
\begin{equation}
\begin{cases}
\mathcal{L}^m_{p,q}(u)=\lambda|u|^{r-2}u+|u|^{p^*-2}u & \text{in } \Omega ,\\
u=\nabla u=\cdots\nabla^{m-1} u=0
    &\text{on} ~{\partial\Omega},
\end{cases}
\end{equation}
admits at least one nontrivial weak solution in $W_0^{m,\mathcal{H}}(\Omega)$ under the same hypotheses as in Theorem \ref{thm3.2.1}.\end{remark}
Furthermore, we are going to discuss the following problem
\begin{equation}\label{3.2.5}\tag{$\mathcal{P}_c$}
\begin{cases}
\mathcal{L}^m_{p,q}(u) = c(x)|u|^{s-2}u + \mu|u|^{p^*-2}u + b(x)|u|^{q^*-2}u & \text{in } \Omega, \\
u=\nabla u=\cdots\nabla^{m-1} u=0,
    &\text{on  }{\partial\Omega},
\end{cases}
\end{equation}
where $p^* \leq s < q^*$, $\mu \geq 0$, $b(x) \in C(\overline{\Omega}) \setminus \{0\}$, and $c(x) \in L^{\infty}(\Omega)$ are nonnegative functions satisfying the hypotheses \eqref{3.1.2} and \eqref{3.1.3}. 
\noindent
\begin{theorem}[\bf The Existence Result-II]\label{thm3.2.4} 
Assume that hypotheses \eqref{3.1.2} and \eqref{3.1.3} hold, and there exists a ball \(B_{\sigma}(x_{0}) \subset \Omega\) such that
\begin{equation}\label{3.2.6}
a(x) = a_{0}, \quad b(x) = \|b\|_\infty > 0, \quad\text{and}\quad c(x) \geq c_{0} \quad \text{for a.a. } x \in B_{\sigma}(x_{0}),
\end{equation}
for some constant \(c_{0} > 0\). Then there exists \(\mu^{*} > 0\) such that problem \eqref{3.2.5} admits at least one nontrivial weak solution in \(W_{0}^{m, \mathcal{H}}(\Omega)\) for all \(0 \leq \mu < \mu^{*}\) and under the assumptions described below
\begin{enumerate}
    \item[\normalfont{(a)}] \(1 < p < \frac{N(q-1)}{N-m}\) and \(\frac{N^{2}(q-1)}{(N-m)(N-mq)} < s < q^{*}\),
    \item[\normalfont{(b)}] \(\frac{N(q-1)}{N-m} \leq p < q\) and \(\frac{N p}{N-mq} < s < q^{*}\).
\end{enumerate}
\end{theorem}
\begin{remark}
 Specifically for \(\mu=0\), Theorem \ref{thm3.2.4} gives the conditions under which the following problem
\begin{equation}\tag{$\mathcal{P}$}
\begin{cases}
    \mathcal{L}^m_{p,q}(u)  = c(x)|u|^{s-2}u + b(x)|u|^{q^{*}-2}u & \text{in } \Omega, \\
    u=\nabla u=\cdots\nabla^{m-1} u=0
    &\text{on  }{\partial\Omega}, 
\end{cases}
\end{equation}
possesses at least one nontrivial weak solution in \(W_{0}^{m, \mathcal{H}}(\Omega)\), which is also a problem of independent interest.
\end{remark}
\subsection{Nonexistence of solutions to the main problem via Poho\v{z}aev identity}
It is well known that nonexistence results for nonlinear elliptic PDEs
with critical and supercritical nonlinearities often rely on the
Poho\v{z}aev identity, which was first established by the Soviet
mathematician S.~Poho\v{z}aev \cite{Pohozaev-1965} in 1965. For detailed studies, we refer to the works of Villegas \cite{Villegas-2013}, Dolbeault-Stanczy \cite{Dolbeault-Stanczy-2010} and the references therein. Such identities frequently appear in the
study of Yamabe-type problems in differential geometry, harmonic maps,
control theory, and geometric analysis.

We are now in a position to derive a Poho\v{z}aev identity in the
higher-order setting.
\begin{theorem}[\bf The Poho\v{z}aev Type Identity]\label{thm3.3.1} Let $\Omega \subset \mathbb{R}^N$ be a smooth bounded domain with Lipschitz boundary $\partial\Omega$. Furthermore, let $1<p<q<\frac{N}{m}$, $(N-1){q}\leq{N}p$, and $a(\cdot) \in C^1(\Omega)$ be such that $a(x)\ge 0$ for a.a. $x\in\Omega$. If $u \in W^{m,\mathcal{H}}_0(\Omega) \cap W^{m+1,\mathcal{H}}(\Omega)$ is a weak solution of \eqref{3.1.1}, then the following identity holds
\begin{equation}\label{eqpohozaev}
\begin{aligned}
 \left( \frac{1}{p} - \frac{1}{q} \right) &
\int_{\Omega} |\nabla^m u|^p  \,\mathrm{d}x 
+ \frac{1}{Nq} \int_{\Omega} |\nabla^m u|^q (\nabla a \cdot x) \,\mathrm{d}x  \\
& + \frac{1}{N} \int_{\partial \Omega} \Bigl(
\Bigl( 1 - \frac{1}{p} \Bigr) |\nabla^m u|^p
+ \Bigl( 1 - \frac{1}{q} \Bigr) a(x) |\nabla^m  u|^q
\Bigr) (x \cdot \nu) \, d\sigma 
= \int_{\Omega} \Bigl( F(x,u) - \frac{1}{q^*} f(x,u)u \Bigr) \,\mathrm{d}x,
\end{aligned}
\end{equation}
where $F(x,t) := \int_0^t f(x,\tau)\,d\tau$ and $\nu$ denotes the outward unit normal on $\partial \Omega$.
\end{theorem}
By using Theorem \ref{thm3.3.1}, we have the following nonexistence result.
\begin{theorem}[\bf The Nonexistence Result-I]\label{thm3.3.2}
Let $\Omega \subset \mathbb{R}^N$ be a bounded and star-shaped domain with Lipschitz boundary $\partial\Omega$. In addition, assume that $1<p<q<\frac{N}{m}$, $\frac{q}{p}\leq\frac{N}{N-1}$ and $a(\cdot) \in C^1(\Omega)$ is radial, radially nondecreasing, and satisfies $a(x)\ge 0$ for a.a. $x\in\Omega$. Define the nonlinear term $f\colon\Omega\times\mathbb{R}\to \mathbb{R}$ by
\[
f(x,t) = c(x) |t|^{r-2}t + \mu |t|^{p^*-2}t + b(x)|t|^{q^*-2}t\quad \text{for a.a. } x \in \Omega~\text{and for all }t \in \mathbb{R},
\] 
with $p \le r < q^*$, $\mu \le 0$, and $b,c \in L^\infty(\Omega)$. Then problem \eqref{3.1.1} admits no nontrivial weak solution $u \in W^{m,\mathcal{H}}_0(\Omega)\cap W^{m+1,\mathcal{H}}(\Omega)$ under the hypotheses listed below
\begin{itemize}
  \item[\normalfont{(a)}] $c(x)\le 0$ a.e. in $\Omega$,
  \item[\normalfont{(b)}] $r = p$ and $0 < c_\infty < \dfrac{\lambda_1(p)\,N(q-p)}{N(q-p) + mpq}\quad\text{with}\quad c_\infty = \|c\|_{\infty}.$
\end{itemize}
\end{theorem}
Moreover, we also have the following result based on the Poho\v{z}aev-type identity defined above.
 \begin{theorem}[\bf The Nonexistence Result-II]\label{thm3.3.3}
 Let $\Omega \subset \mathbb{R}^N$ be a bounded, star-shaped domain with Lipschitz boundary $\partial\Omega$, $1<p<q<\frac{N}{m}$, $\frac{q}{p}\leq\frac{N}{N-1}$ and $a(x) \in C^1(\Omega)$ be radial, radially nondecreasing, and satisfies  $a(x)\ge 0$ for a.a. $x\in\Omega$. Define the nonlinear term $f\colon\Omega\times\mathbb{R}\to \mathbb{R}$ by
\[
f(x,t) = c(x) |t|^{r-2}t + \mu |t|^{p^*-2}t + b(x)|t|^{q^*-2}t\quad \text{for a.a. } x \in \Omega~\text{and for all }t \in \mathbb{R},
\]
where $\mu \ge 0$, $0 \le b(x)\in L^\infty(\Omega)$ satisfying \eqref{3.1.2}, and $0 \le c(x)\in C(\overline{\Omega})\setminus\{0\}$. Then there exists a positive constant $\kappa = \kappa(\Omega, p, q, r, a, c, \mu, b_{\infty})$
such that problem \eqref{3.1.1} admits no weak solution 
$u \in W^{m,\mathcal{H}}_0(\Omega)\cap W^{m+1,\mathcal{H}}(\Omega)$ 
with $\|u\|\le \kappa$ based on the hypotheses defined below
\begin{itemize}
  \item[\normalfont{(a)}] $p < r \le p^*$,
  \item[\normalfont{(b)}] $p^* < r < q^*$ and $c(x)$ satisfies \eqref{3.1.3},
  \item[\normalfont{(c)}] $p^* < r < q^*$ and $a'_0=\inf_{x\in\operatorname{supp}(c)} a(x) > 0$.
\end{itemize}
\end{theorem}

\section{Preliminaries}\label{prelim}
In this section, we establish the fundamental framework of Musielak-Orlicz spaces that will serve as the mathematical foundation of our analysis as well as some essential definitions, properties, and structural results that are crucial for the study of our main problem. For comprehensive treatments of these topics, we refer the reader to \cite{Musielak-1983,Diening-Harjulehto-Hasto-Ruzicka-2011}.
\begin{definition}
Let $\varphi : [0,\infty) \to [0,\infty)$ be a function. Then it is called a $\Phi$-function, if the following properties hold:
\begin{itemize}
    \item[\normalfont{(a)}] $\varphi$ is continuous and convex,
    \item[\normalfont{(b)}] $\varphi(0) = 0$,
    \item[\normalfont{(c)}]$\varphi(t) > 0$ for all $t > 0$.
\end{itemize}
\end{definition}
\begin{definition}
    Let $\Omega \subset \mathbb{R}^n$ be a bounded domain. A function $\varphi  :\Omega \times [0,\infty) \to [0,\infty)$ is called a generalized $\Phi$-function, denoted by $\varphi \in \Phi(\Omega)$, if it satisfies
\begin{itemize}
    \item[\normalfont{(a)}] For each fixed $t \geq 0$, the mapping $\cdot \mapsto \varphi(\cdot, t)$ is measurable.
    \item[\normalfont{(b)}] For a.a. $x \in \Omega$, the mapping $\cdot \mapsto \varphi(x, \cdot)$ is a $\Phi$-function.
\end{itemize}
\end{definition}
We say $\varphi \in \Phi(\Omega)$ is locally integrable, if $\varphi(\cdot, t)$ belongs to $L^1(\Omega)$ for every $t > 0$. Moreover, $\varphi$ is said to satisfy the ($\Delta_2$)-condition, if there exists a constant $C>0$ and a nonnegative function $h \in L^1(\Omega)$ such that
\[
    \varphi(x, 2t) \leq C \varphi(x, t) + h(x)\quad \text{for a.a. } x \in \Omega~~\text{and for all } t \geq 0.
\]
Now, for any $\varphi \in \Phi(\Omega)$ satisfying the $(\Delta_2$)-condition, we define the Musielak-Orlicz space $L^\varphi(\Omega)$ as the collection of all measurable functions $u : \Omega \to \mathbb{R}$ for which the modular
\[
    \rho_\varphi(u) = \int_\Omega \varphi(x, |u(x)|) \, \mathrm{d}x<+\infty.
\]
It forms a reflexive Banach space with the Luxemburg norm
\[
    \|u\|_\varphi = \inf \left\{ \gamma > 0 : \rho_\varphi\left(\frac{u}{\gamma}\right) \leq 1 \right\}.
\]
\begin{proposition}[Unit Ball Property]
Suppose that $\varphi \in \Phi(\Omega)$ and $u \in L^\varphi(\Omega)$. Then the following results hold:
\begin{itemize}\label{eq3.3.1}
    \item[\normalfont{(a)}] $\rho_{\varphi}(u) < 1 \ \text{if and only if}~ \|u\|_{\varphi} < 1$,
    \item[\normalfont{(b)}] $\rho_{\varphi}(u) = 1 \ \text{if and only if}~ \|u\|_{\varphi} = 1,$
    \item[\normalfont{(c)}] $\rho_{\varphi}(u) > 1 \ \text{if and only if}~ \|u\|_{\varphi} > 1$.
\end{itemize}
\end{proposition}

\begin{definition}
Let $\varphi \in \Phi(\Omega)$. The Musielak--Orlicz--Sobolev space
$W^{m,\varphi}(\Omega)$ consists of all functions $u \in L^\varphi(\Omega)$
such that $|\nabla^k u| \in L^\varphi(\Omega)$ for all
$k\in\{0,1,2,\ldots,m\}$. This space is endowed with the norm
\[
    \|u\|_{m,\varphi}
    =
    \sum_{k=0}^{m}\|\nabla^k u\|_\varphi,
\]
where $\|\nabla^k u\|_\varphi$ denotes the Luxemburg norm of
$|\nabla^k u|$.
\end{definition}
\begin{definition}
A function $\varphi : [0,\infty) \to [0,\infty)$ is called $\mathcal{N}$-function, if in addition to being a $\Phi$ function, it satisfies the growth conditions
\[
    \lim_{t \to 0^+} \frac{\varphi(t)}{t} = 0 \quad \text{and} \quad \lim_{t \to \infty} \frac{\varphi(t)}{t} = \infty.
\]
\end{definition}
\begin{definition}
Let $\Omega \subset \mathbb{R}^n$ be a bounded domain. A function $\varphi:  \Omega \times \mathbb{R} \to [0,\infty)$ is called a generalized $\mathcal{N}$-function,  denoted by $\varphi \in N(\Omega)$, if it satisfies
\begin{itemize}
    \item[\normalfont{(a)}] For each fixed $t\in\mathbb{R}$, the mapping $\cdot \mapsto \varphi(\cdot, t)$ is measurable.
    \item[\normalfont{(b)}] For a.a. $x \in \Omega$, the mapping $\cdot \mapsto \varphi(x, \cdot)$ is a $\mathcal{N}$-function.
\end{itemize}
\end{definition}
Further, let $\varphi \in N(\Omega)$ be such that it is locally integrable. Then we define 
$$W_0^{m,\varphi}(\Omega)={\overline{C_0^\infty(\Omega)}}^{\|\cdot\|_{m,\varphi}} .$$
To analyze the relationship between different Musielak-Orlicz spaces, we introduce the following comparison concepts and embedding results.

\begin{definition}
Assume that $\varphi, \psi \in \Phi(\Omega)$. We say that $\varphi$ is \emph{dominated by} $\psi$, written $\varphi \preceq \psi$, if there exist positive constants $C_1, C_2$ and a nonnegative function $h \in L^1(\Omega)$ such that
\[
    \varphi(x, t) \leq C_1 \psi(x, C_2 t) + h(x)\quad \text{for a.a. } x \in \Omega~~\text{and for all } t \geq 0.
\]
\end{definition}
\begin{proposition}\label{pro3.3.6}\cite[Theorem 8.5]{Musielak-1983}
Let $\varphi, \psi \in N(\Omega)$ be such that $\varphi \preceq \psi$. Then the space $L^\psi(\Omega)$ is continuously embedded in $L^\varphi(\Omega)$, which is denoted by $L^\psi(\Omega) \hookrightarrow L^\varphi(\Omega)$.
\end{proposition}
\begin{definition}
Assume that $\varphi, \psi \in N(\Omega)$. We say that $\phi$ has \emph{essentially slower growth} than $\psi$ near infinity, denoted $\phi \ll \psi$, if for every $k > 0$, there holds
\[
    \lim_{t \to \infty} \frac{\phi(x, kt)}{\psi(x, t)} = 0 \quad \text{ uniformly for a.a.} x \in \Omega.
\]
\end{definition}

Now, we focus on the suitable function space to study our double phase polyharmonic problem \eqref{3.1.1}. For this, we first introduce the double phase function $\mathcal{H} : \Omega \times [0,\infty) \to [0,\infty)$ given by
\begin{equation}\label{def3.3.2}
\mathcal{H}(x,t) = t^p + a(x)t^q \quad \text{ for all } (x,t) \in \Omega \times [0,\infty),
\end{equation}
where $1 < p < q < \frac{N}{m}$, $(N-1)q\leq Np$, and $0 \leq a(\cdot) \in L^\infty(\Omega) \cap C^{0,\frac{N}{p(q-p)}}(\Omega)$. This function represents a locally integrable generalized $\mathcal{N}$-function with the $(\Delta_2)$-condition. Then the Musielak-Orlicz Lebesgue space $L^\mathcal{H}(\Omega)$ consists of all measurable functions $u  :\Omega \to \mathbb{R}$ such that 
\[
\rho_\mathcal{H}(u) = \int_\Omega \mathcal{H}(x,|u|) \, \mathrm{d}x < +\infty.
\]
It also forms a reflexive Banach space with the Luxemburg norm
\[
\|u\|_\mathcal{H} = \inf\left\{\gamma > 0:  \rho_\mathcal{H}\left(\frac{u}{\gamma}\right) \leq 1\right\}.
\]
 By Proposition \ref{eq3.3.1}, we obtain
\begin{align}\label{eq3.3.2}
\min\{\|u\|_\mathcal{H}^p, \|u\|_\mathcal{H}^q\} \leq \int_\Omega (|u|^p + a(x)|u|^q) \, \mathrm{d}x \leq \max\{\|u\|_\mathcal{H}^p, \|u\|_\mathcal{H}^q\},\quad \forall~ u \in L^\mathcal{H}(\Omega). 
\end{align}
For any $\mathcal{N}$-function of the type $\varphi(x,t) = t^\alpha + \psi(x)t^\beta$ with $1 < \alpha < \beta$, there holds
\begin{align}\label{eq3.3.3}
\min\{\|u\|_\varphi^\alpha, \|u\|_\varphi^\beta\} \leq \int_\Omega (|u|^\alpha + \psi(x)|u|^\beta) \, \mathrm{d}x \leq \max\{\|u\|_\varphi^\alpha, \|u\|_\varphi^\beta\},\quad \forall~ u \in L^\varphi(\Omega). 
\end{align}

\begin{definition}\label{definitionsobolev}
With $\mathcal H$ as defined in \eqref{def3.3.2}, the Musielak--Orlicz--Sobolev space
associated with the double phase function $\mathcal H$ is defined as the particular case of
$W^{m,\varphi}(\Omega)$ corresponding to $\varphi=\mathcal H$.
Moreover, $W_0^{m,\mathcal H}(\Omega)$ denotes the closure of
$C_0^\infty(\Omega)$ in $W^{m,\mathcal H}(\Omega)$ with respect to the norm
$\|\cdot\|_{m,\mathcal H}$.
\end{definition}

\begin{remark}
Since $\mathcal H$ satisfies the required structural assumptions, the space
    $W^{m,\mathcal H}(\Omega)$ is a reflexive Banach space. Consequently,
    $W_0^{m,\mathcal H}(\Omega)$ is also reflexive, being a closed subspace of
    $W^{m,\mathcal H}(\Omega)$.
\end{remark}

The following inequality is known as the Poincar\'{e} inequality. For the detailed proof, we refer to Colasuonno-Squassina \cite[Proposition 2.18]{Colasuonno-Squassina-2016} and Adams \cite{Adams-2003}.
\begin{theorem}
If domain $\Omega\subset \mathbb{R}^N$ has a finite width, then there exists a constant $k=k(p)>0$ such that 
\begin{equation}\label{poincare}
    \|\phi\|_\mathcal{H}\leq k\|\nabla\phi\|_{\mathcal{H}},~\forall~\phi\in C_0^{\infty}(\Omega).
\end{equation}
 \end{theorem}
\begin{corollary}\label{poincarecoro}
    If $\Omega$ has finite width, then $\|\nabla^m\phi\|_{\mathcal{H}}$ is a norm on $W_0^{m,\mathcal{H}}(\Omega)$
 which is equivalent to the standard norm $\|\phi\|_{m,\mathcal{H}}$.  
\end{corollary}
\begin{proof}
    If $\phi\in C_0^{\infty}(\Omega)$, then any derivative of $\phi$ also belong to $C_0^{\infty}(\Omega)$. Now, by using  \eqref{poincare}, we have
$$ \|\nabla\phi\|_{\mathcal{H}}\leq\|\phi\|_{1,\mathcal{H}}=\|\nabla\phi\|_{\mathcal{H}}+\|\phi\|_{\mathcal{H}}\leq (1+k)\|\nabla\phi\|_{\mathcal{H}},$$
 for some suitable constant $k=k(p)>0$. Notice that successive iteration of this inequality applied to derivatives $\nabla^\gamma\phi$ for all $\gamma\leq m-1$ gives 
$$\|\nabla^m\phi\|_{\mathcal{H}}\leq\|\phi\|_{m,\mathcal{H}}\leq k_1\|\nabla^m\phi\|_{\mathcal{H}},~\phi\in C_0^{\infty}(\Omega),$$
for some suitable constant $k_1=k_1(p)>0$. On exploiting the density argument of $C_0^{\infty}(\Omega)$ in  $W_0^{m,\mathcal{H}}(\Omega)$, the result immediately follows.
\end{proof}
By Corollary \ref{poincarecoro}, one can easily define an equivalent norm $\|u\| = \|\nabla^m u\|_\mathcal{H}$ for the space $W_0^{m,\mathcal{H}}(\Omega)$. Since reflexivity is preserved under equivalent norm and thus, we deduce that $(W_0^{m,\mathcal{H}}(\Omega),\|\cdot\|)$ is again a reflexive Banach space. Now, by applying \eqref{eq3.3.2}, the following holds
\begin{align}\label{eq3.3.4}
{\min\{\|u\|^p, \|u\|^q\} \leq \int_\Omega (|\nabla^m u|^p + a(x)|\nabla^m u|^q) \, \mathrm{d}x \leq \max\{\|u\|^p, \|u\|^q\} },~\forall ~u \in W_0^{m,\mathcal{H}}(\Omega).
\end{align}
The following embeddings can be directly obtained from Colasuonno-Squassina \cite[Proposition 2.15]{Colasuonno-Squassina-2016}. For simplicity, we provide here a short proof.
\begin{proposition}\label{pro3.3.9}
Assume that {$1<p<q<\frac{N}{m}$} and $0\le a(\cdot)\in L^\infty(\Omega)$.  Define the weighted Lebesgue space as follows
\[
L^q_a(\Omega)
= \Bigl\{u:\Omega\to\mathbb{R} \text{ measurable}\colon\int_\Omega a(x)\,\lvert u\rvert^q\,\mathrm{d}x<+\infty\Bigr\}
\]
with seminorm
\[
\lvert u\rvert_{q,a}
=\biggl(\int_\Omega a(x)\,\lvert u\rvert^q\,\mathrm{d}x\biggr)^{1/q},
\]
which turns into a norm if $a(x)>0$ for a.a. $x\in \Omega$. Then the following results hold:
\begin{enumerate}
  \item[\normalfont{(a)}] The embeddings $L^\mathcal{H}(\Omega)\hookrightarrow L^r(\Omega)$ and $W^{m,\mathcal{H}}_0(\Omega)\hookrightarrow W^{m,r}_0(\Omega)$ for all $r\in[1,p]$ are continuous.
  \item[\normalfont{(b)}] The embedding $W^{m,\mathcal{H}}_0(\Omega)\hookrightarrow L^r(\Omega)$ is continuous for $r\in[1,p^*]$ and compact for $r\in[1,p^*)$.
  \item[\normalfont{(c)}] The chain of continuous embeddings $L^q(\Omega)\hookrightarrow L^\mathcal{H}(\Omega)\hookrightarrow L^q_a(\Omega)$ hold.
\end{enumerate}
\end{proposition}
\begin{proof}
    Define $\mathcal{H}_p(x, t) = t^p$ for a.a. $x \in \Omega$ and for all $t\geq 0$. By Proposition \eqref{pro3.3.6}, we obtain that the embeddings $L^\mathcal{H}(\Omega) \hookrightarrow L^p(\Omega)$ and $W_0^{m,\mathcal{H}}(\Omega) \hookrightarrow W_0^{m,p}(\Omega)$ are continuous, thanks to $\mathcal{H}_p \preceq \mathcal{H}$ and thus, the statement of (a) follows. The proof of statement of (b) trivially holds due to embedding of $W_0^{m,p}(\Omega)$ and omitted here. Finally, observe that
\[
\int_{\Omega} a(x) |u|^q \mathrm{d}x \leq \int_{\Omega} (|u|^p + a(x)|u|^q) \mathrm{d}x = \rho_\mathcal{H}(u)
\]
holds for all $u \in L^\mathcal{H}(\Omega)$. Moreover, for $u \neq 0$, we have
\[
\int_{\Omega} a(x) \left| \frac{u}{\|u\|_\mathcal{H}} \right|^q \mathrm{d}x \leq 1 \iff \lvert u\rvert_{q,a} \leq \|u\|_\mathcal{H}.
\]
By the hypothesis, we deduce that
\[
\mathcal{H}(x, t)\leq 1 + (1 + \|a\|_{\infty}) t^q, \quad\text{for a.a.}~ x \in \Omega ~\text{and for all}~ t\geq 0.
\]
The above inequality with Proposition \eqref{pro3.3.6} implies the validity of statement of (c). This ends the proof. 
\end{proof}
\begin{definition}\label{def4.8}
For each $x\in\Omega$, we define the Sobolev conjugate $\mathcal{H}_*$ of $\mathcal{H}$ by
\[
\mathcal{H}_*^{-1}(x, s) = \int_0^s \frac{\mathcal{H}^{-1}(x, \tau)}{\tau^{1 + 1/N}}  d\tau.
\]
Similarly, the $m$-th Sobolev conjugate $\mathcal{H}_{*,m}$ with $m \geq 1$ is defined recursively by
\begin{equation}
    \begin{cases}
      \mathcal{H}_{*,0}(x, t) = \mathcal{H}(x, t),\\
      \mathcal{H}_{*,k+1}^{-1}(x, s) = \int_0^s \frac{\mathcal{H}_{*,k}^{-1}(x, \tau)}{\tau^{1 + 1/N}}  d\tau \quad \text{for all }~ k \in\{ 0, 1, \dots, m-1\}.
    \end{cases}
\end{equation}
\end{definition}
Let $\Omega \subset \mathbb{R}^N$ be a smooth bounded domain with Lipschitz boundary $\partial\Omega$ and let $\mathcal{H}: \overline{\Omega} \times [0, \infty) \to [0, \infty)$ be a Musielak-Orlicz function satisfying the hypotheses for all $k\in\{ 1,2, \dots, m\}$ as follows:
\begin{itemize}
\item[\normalfont{(P1)}] There holds $\mathcal{H}\in\Phi(\Omega)$.

\item[\normalfont{(P2)}]
$\mathcal{H}$ is continuous on $\overline{\Omega} \times [0, \infty)$ and $\mathcal{H}(x, t) \in (0, \infty)$ for a.a. $x \in \overline{\Omega}$ and $t > 0$.

\item[\normalfont{(P3)}]
$\mathcal{H}(x, t) = \mathcal{H}(x, 1)t$ for a.a. $x \in \overline{\Omega}$ and $t \in [0, 1)$.

\item[\normalfont{(P4)}]
For each $k$, the function $T_k(x) = \lim_{s \to \infty} \mathcal{H}_{*,k}^{-1}(x, s)$ is continuous on $\overline{\Omega}$ and $T_k\in C^{1,0}(\text{Dom}(T_k))$, where $C^{1,0}(\text{Dom}(T_k))$ is the set of all locally Lipschitz continuous functions on $\text{Dom}(T_k)$.

\item[\normalfont{(P5)}]
For each $k$, $\mathcal{H}_{*,k} \in C^{1,0}(\text{Dom}(\mathcal{H}_{*,k}))$ and there exist constants $0<\delta_0 < \frac{1}{N}$, $C_0 >0$, and $0<t_0 < \min_{x \in \overline{\Omega}} T_k(x)$ such that for all $x \in \Omega$ and $t \in [t_0, T_k(x))$, there holds 
\[
\left| \frac{\partial \mathcal{H}_{*,k}(x, t)}{\partial x_j} \right| \leq C_0 (\mathcal{H}_{*,k}(x, t))^{1 + \delta_0} \quad \text{for all}\quad j\in \{ 1, \dots, N\},
\]
provided that $\frac{\partial \mathcal{H}_{*,k}(x, t)}{\partial x_j}$ exists for each $j\in \{ 1, \dots, N\}$.
\end{itemize}
\begin{theorem}
Assume that hypotheses \textnormal{(P1)-(P5)} hold. Then the embedding described below
\[
W^{m,\mathcal{H}}(\Omega) \hookrightarrow L^{\mathcal{H}_{*,m}}(\Omega)
\]
is continuous. Moreover, if there holds $B \ll \mathcal{H}_{*,m}$, the embedding $W^{m,\mathcal{H}}(\Omega) \hookrightarrow L^B(\Omega)$ is compact, where $B$ is a Musielak-Orlicz function. In particular, the embedding $W^{m,\mathcal{H}}(\Omega)\hookrightarrow L^\mathcal{H}(\Omega)$ is compact. 
\end{theorem}
\begin{proof}
The proof of our arguments  rely on  Fan \cite[Theorems 1.1 and 1.2]{Fan-2012}, Colasuonno-Squassina \cite[Proposition 2.18]{Colasuonno-Squassina-2016} and Donaldson-Trudinger \cite[Theorem 3.9]{Donaldson-Trudinger-1971}. Here, we present an outline of the proof in the higher-order Sobolev setting. For $m = 1$, by using result obtained by Fan \cite[Theorem 1.1a]{Fan-2012}, there exists a constant $C_1 > 0$ such that 
\[
\|u\|_{\mathcal{H}_{*}} \leq C_1 \|u\|_{W^{1,\mathcal{H}}},~\forall~u \in W^{1,\mathcal{H}}(\Omega).
\]
Thus, the embedding $W^{1,\mathcal{H}}(\Omega) \hookrightarrow L^{\mathcal{H}_{*}}(\Omega)$ holds.
Further, by adapting the mathematical induction,  we assume that the embedding $W^{m,\mathcal{H}}(\Omega) \hookrightarrow L^{\mathcal{H}_{*,m}}(\Omega)$ with $m>1$ holds. Thus, there exists a constant $C_m > 0$ such that
\begin{equation}\label{1}
\|u\|_{\mathcal{H}_{*,m}} \leq C_m \|u\|_{W^{m,\mathcal{H}}},~\forall~u \in W^{m,\mathcal{H}}(\Omega). 
\end{equation}
Now, let $u \in W^{m+1,\mathcal{H}}(\Omega)$, then $u \in W^{m,\mathcal{H}}(\Omega)$ and $\nabla u \in W^{m,\mathcal{H}}(\Omega)$. By using \eqref{1} as the induction hypothesis, we have
\begin{equation}\label{2}
\|u\|_{\mathcal{H}_{*,m}} \leq C_m \|u\|_{W^{m,\mathcal{H}}} \leq C_m \|u\|_{W^{m+1,\mathcal{H}}} 
\end{equation}
and
\begin{equation}\label{3}
\|\nabla u\|_{\mathcal{H}_{*,m}} \leq C_m \|\nabla u\|_{W^{m,\mathcal{H}}} \leq C_m \|u\|_{W^{m+1,\mathcal{H}}}. 
\end{equation}
Since $u \in L^{\mathcal{H}_{*,m}}(\Omega)$ and $\nabla u \in L^{\mathcal{H}_{*,m}}(\Omega)$, it follows that $u \in W^{1,\mathcal{H}_{*,m}}(\Omega)$. Moreover, from \eqref{2} and \eqref{3}, we obtain
\begin{equation}\label{4}
\|u\|_{W^{1,\mathcal{H}_{*,m}}} = \|u\|_{\mathcal{H}_{*,m}} + \|\nabla u\|_{\mathcal{H}_{*,m}} \leq 2C_m  \|u\|_{W^{m+1,\mathcal{H}}}. 
\end{equation}
Now, since $\mathcal{H}_{*,m}$ satisfies (P1)--(P5), we can  apply \cite[Theorem 1.1(i)]{Fan-2012} to the space $W^{1,\mathcal{H}_{*,m}}(\Omega)$. Thus, there exists a constant $C_{*,m} > 0$ such that 
\begin{equation}\label{5}
\|v\|_{(\mathcal{H}_{*,m})_*} \leq C_{*,m} \|v\|_{W^{1,\mathcal{H}_{*,m}}},\quad\forall~v \in W^{1,\mathcal{H}_{*,m}}(\Omega) 
\end{equation}
From the definition of $(\mathcal{H}_{*,m})$ , we have $(\mathcal{H}_{*,m})_* = \mathcal{H}_{*,m+1}$. By applying \eqref{5} to $u$ and using \eqref{4}, one has
\begin{equation*}\label{6}
\|u\|_{\mathcal{H}_{*,m+1}} \leq C_{*,m} \|u\|_{W^{1,\mathcal{H}_{*,m}}} \leq 2C_{*,m} C_m  \|u\|_{W^{m+1,\mathcal{H}}}. 
\end{equation*}
This shows that the embedding $W^{m+1,\mathcal{H}}(\Omega) \hookrightarrow L^{\mathcal{H}_{*,m+1}}(\Omega)$ is continuous.
Hence, the embedding holds for all $m \geq 1$.
Now, suppose that $B$ is a Musielak-Orlicz function such that $B \ll \mathcal{H}_{*,m}$, that is, for all $k > 0$, we have
\[
\lim_{t \to \infty} \frac{B(x, k t)}{\mathcal{H}_{*,m}(x, t)} = 0 \quad \text{uniformly for } x \in \Omega.
\]
It follows that the embedding $W^{m,\mathcal{H}}(\Omega) \hookrightarrow L^B(\Omega)$ is compact.
Indeed, first note that from the previous argument given in \eqref{4}, we have the following continuous embedding
 \begin{equation}\label{9}
 W^{m,\mathcal{H}}(\Omega) \hookrightarrow W^{1,\mathcal{H}_{*,m-1}}(\Omega).
 \end{equation}
Now, since $B \ll \mathcal{H}_{*,m} = (\mathcal{H}_{*,m-1})_*$, by using the result of Fan \cite[Theorem 1.1(ii)]{Fan-2012}, again the embedding $W^{1,\mathcal{H}_{*,m-1}}(\Omega) \hookrightarrow L^B(\Omega)$ is compact. Combining this with  \eqref{9}, we obtain that the embedding $ W^{m,\mathcal{H}}(\Omega) \hookrightarrow  L^B(\Omega)$ is compact. This finishes the proof.
\end{proof}
\begin{proposition}\label{prop3.3.12}
Let $\mathcal{C}(x,t)=t^r + c(x)\,t^s$ for a.a. $x\in\Omega$ and $t\geq0$ be such that $1<r<p^*\le s<q^*$, $\quad c\in L^\infty(\Omega),$ and $c\ge0$ satisfying condition \eqref{3.1.3}. Then the embedding $W^{m,\mathcal{H}}_0(\Omega)\hookrightarrow L^\mathcal{C}(\Omega)$ is compact.
\end{proposition}
\begin{proof}
By \cite[Theorem 3.7]{Cianchi-Diening-2024} and \cite[Proposition 3.7]{Ho-Winkert-2023}, it suffices to show
\begin{equation}\label{eq3.3.5}
\lim_{t\to\infty}\frac{\mathcal{C}(x,k\,t)}{t^{p^*}+a(x)^{\,\frac{q^*}{q}}\,t^{q^*}}=0
\quad\text{uniformly for a.a. }x\in\Omega,~\text{and} ~ k>0.
\end{equation}
Due to Young’s inequality for $\epsilon>0$ and \eqref{3.1.3}, we get
\begin{align*}
\frac{\mathcal{C}(x, kt)}{t^{p^*} + a(x)^{\frac{q^*}{q}}t^{q^*}} 
&\leq \frac{k^r t^r + Ca(x)^{s/q}k^s t^s}{t^{p^*} + a(x)^{\frac{q^*}{q}}t^{q^*}} \\
&\leq (1 + k^s) \frac{\epsilon t^{p^*} + \epsilon^{r/(r - p^*)} + C[\epsilon a(x)^{\frac{q^*}{q}}t^{q^*} + \epsilon^{s/(s - q^*)} ]}{t^{p^*} + a(x)^{\frac{q^*}{q}}t^{q^*}} \\
&\leq (1 + k^s) \max\{1, C\} \left( \epsilon + \frac{\epsilon^{r/(r - p^*)} + \epsilon^{s/(s - q^*)} }{t^{p^*}} \right) \\
&
\rightarrow (1 + k^s) \max\{1, C\}\epsilon \quad \text{as }\quad t \rightarrow +\infty.
\end{align*}
By the arbitrariness of $\epsilon$, we conclude that \eqref{eq3.3.5} holds. This finishes the proof.
\end{proof}

Now, we take a nonnegative function \( b(x) \in L^\infty(\Omega) \) satisfying $a_0 = \inf_{x \in \operatorname{supp}(b)} a(x) > 0$. Define the Musielak-Orlicz space \(L^{\mathcal{B}}(\Omega)\) with $\mathcal{B}(x, t) = t^{p^*} + b(x)t^{q^*}$ for all $ (x, t) \in \Omega \times [0, \infty)$.
\begin{proposition}\label{prop3.4.13}
The embedding \(W^{m,\mathcal{H}}_0(\Omega) \hookrightarrow L^{\mathcal{B}}(\Omega)\) is continuous. Moreover, there hold
\begin{align}\label{eq3.3.6}
\int_{\Omega} |u|^{p^*} \mathrm{d}x &\leq \frac{1}{S_p^{\frac{p^*}{p}}} \left( \int_{\Omega} |\nabla^m u|^p \mathrm{d}x \right)^{\frac{p^*}{p}},~\forall~ u \in W^{m,\mathcal{H}}_0(\Omega)  
\end{align}
and 
\begin{align}\label{eq3.3.7}
\int_{\Omega} b(x) |u|^{q^*} \mathrm{d}x &\leq \kappa \left( \int_{\Omega} a(x)|\nabla^m u|^q \mathrm{d}x \right)^{\frac{q^*}{q}},~\forall~ u \in W^{m,\mathcal{H}}_0(\Omega), 
\end{align}
where \( S_p \) is defined in \eqref{eq3.2.13}, \( \kappa = C_s \|b\|_{\infty}a_0^{-\frac{q^*}{q}} \) and \( C_s \) is the best constant in the embedding \( W^{m,q}(\operatorname{supp}(b)) \hookrightarrow L^{q^*}(\operatorname{supp}(b)) \), respectively.
\end{proposition}
\begin{proof}
It follows from Proposition \ref{pro3.3.9} that the embedding \(W^{m,\mathcal{H}}_0(\Omega) \hookrightarrow W_0^{m,p}(\Omega)\) holds. Thus, by using the embedding $W_0^{m,p}(\Omega)\hookrightarrow L^{p^*}(\Omega)$, we infer that  \eqref{eq3.3.6} holds. Now, we assume \( b \not\equiv 0 \), then
\begin{align*}
\int_{\Omega} b(x)|u|^{q^*}\mathrm{d}x 
&\leq \|b\|_{\infty} \int_{\operatorname{supp}(b)} |u|^{q^*}\mathrm{d}x \leq C_s \|b\|_{\infty} \left( \int_{\operatorname{supp}(b)} |\nabla^m u|^q \mathrm{d}x \right)^{\frac{q^*}{q}} \\
&\leq \frac{C_s \|b\|_{\infty}}{a_0^{\frac{q^*}{q}}} \left( \int_{\operatorname{supp}(b)} a(x) |\nabla^m u|^q \mathrm{d}x \right)^{\frac{q^*}{q}} \leq \frac{C_s \|b\|_{\infty}}{a_0^{\frac{q^*}{q}}} \left( \int_{\Omega} a(x) |\nabla^m u|^q \mathrm{d}x \right)^{\frac{q^*}{q}},
\end{align*}
thanks to the embedding $W^{m,q}(supp(b))\hookrightarrow L^{q^*}(supp(b))$. Now, we take a sequence $(u_j)_{j\in\mathbb{N}}\subset W^{m,\mathcal{H}}_0(\Omega)$ such that \( u_j \to u \) in \( W_0^{m,\mathcal{H}}(\Omega) \) as $j\to\infty$ for some $u\in  W_0^{m,\mathcal{H}}(\Omega)$. It follows from \eqref{eq3.3.4} that \( \rho_{\mathcal{H}}(\nabla^m u_j - \nabla^m u) \to 0 \) as $j\to\infty$. Now, by using \eqref{eq3.3.7} and \eqref{eq3.2.13}, we have \( \rho_{\mathcal{B}}(u_j - u) \to 0 \) as $j\to\infty$. Finally, from \eqref{eq3.3.3}, one can see that \( u_j \to u \) in \( L^\mathcal{B}(\Omega) \) as $j\to\infty$. This concludes the proof.
\end{proof}
To explicitly determine the threshold level where the Palais-Smale condition holds for the energy functional, we require an inequality analogous to \eqref{eq3.3.7} but with a constant independent of the support of $b(x)$. This requires stronger regularity assumptions on $b(x)$ leading to the following key estimate.
\begin{proposition}\label{prop3.3.14}
Let \( b (x)\in C(\overline{\Omega}) \) be nonnegative satisfying the hypothesis \eqref{3.1.2}. Then for any \( \epsilon > 0 \), there exists \( C_\epsilon > 0 \) such that 
\begin{equation}\label{eq3.3.8}
\int_{\Omega} b(x) |u|^{q^*} \,\mathrm{d}x \leq \frac{\|b\|_\infty}{(a_0 S_q)^{\frac{q^*}{q}}} 
\Bigg( (1+\epsilon)\!\int_{\Omega} a(x)\,|\nabla^m u|^q \,\mathrm{d}x 
+ C_\epsilon \!\int_{U} |u|^q \,\mathrm{d}x \Bigg)^{\frac{q^*}{q}},~\forall~u \in W_0^{m,\mathcal{H}}(\Omega),
\end{equation}
where \( U = \{x \in \Omega : b(x) > 0\} \), \( \|b\|_\infty = \sup_{x\in\Omega}|b(x)| \), and \( S_q \) is the best Sobolev constant defined in \eqref{eq3.2.13}.
\end{proposition}

\begin{proof}
Let \( u \in C_0^\infty(\Omega) \) and there exists \( \delta>0\) be small enough. Define the following set
\[
U_\delta = \{x\in U:\text{dist}(x,\partial U\setminus\partial\Omega)>\delta\}.
\]
By uniform continuity of \(b(x)\) and \(b(x)\bigl|_{\partial U\setminus\partial\Omega}=0\), without loss of generality, we can choose \(\delta\) so that 
\(\,0 < b(x)\le \delta_1\) on \(U\setminus U_\delta\) for some
\(\delta_1>0\), to be specified later.  Let \(\eta\in C_0^\infty(U,[0,1])\) be such that 
\[
\eta(x)=
\begin{cases}
1,&x\in U_\delta,\\
0,&x\notin U_{\delta/2},
\end{cases}
\quad\text{and}\quad
C_\eta=\max_{0\le k\le m}\|\nabla^k\eta\|_{L^\infty}.
\]
It follows that
\begin{align*}
 \int_\Omega b(x)|u|^{q^*}\mathrm{d}x
&=\int_U b(x)|\eta u|^{q^*}\mathrm{d}x+\int_U b(x)(1-\eta^{q^*})|u|^{q^*}\mathrm{d}x
 \\ &\le \|b\|_\infty\!\int_U|\eta u|^{q^*}\mathrm{d}x+\delta_1\!\int_{U\setminus U_\delta}|u|^{q^*}\mathrm{d}x.   
\end{align*}
Thus, we have the following estimate
\begin{equation}\label{eq3.3.9mod}
\Bigg(\!\int_\Omega b(x)|u|^{q^*}\mathrm{d}x\Bigg)^{\frac{q}{q^\ast}}
\le \|b\|_\infty^{\frac{q}{q^\ast}}\Bigg(\!\int_U|\eta u|^{q^*}\mathrm{d}x\Bigg)^{\frac{q}{q^\ast}}
+\delta_1^{\frac{q}{q^\ast}}\Bigg(\!\int_{U\setminus U_\delta}|u|^{q^*}\mathrm{d}x\Bigg)^{\frac{q}{q^\ast}}.
\end{equation}
Since \(\eta u\in C_0^\infty(U)\), by the Sobolev embedding \(W_0^{m,q}(U)\hookrightarrow L^{q^*}(U)\), one has 
\begin{equation}\label{eq3.3.10mod}
\Bigg(\!\int_U|\eta u|^{q^*}\mathrm{d}x\Bigg)^{\frac{q}{q^\ast}}
\le\frac1{S_q}\int_U|\nabla^m(\eta u)|^q\mathrm{d}x.
\end{equation}
By the Leibniz rule, we also have
\begin{align}\label{eq1.1}
\nabla^m(\eta u)
=\eta\,\nabla^m u+\sum_{k=1}^m\binom mk(\nabla^{m-k}u)(\nabla^k\eta).    
\end{align}
It is known that for any $p>1$ and $\vartheta>0$, there holds
\begin{align}\label{eq1.2}
  (s+t)^p\leq (1+\vartheta)s^p+C_\vartheta t^p,~\forall~s,t>0,  
\end{align}
where $C_\vartheta=\big(1-(1+\vartheta)^{-\frac{1}{p-1}}\big)^{1-p}$. By using \eqref{eq1.1}, \eqref{eq1.2} along with the fact that \(|\nabla^k\eta|\le C_\eta\) in $U$, we get
\begin{align}\label{eq1.3}
 \int_U|\nabla^m(\eta u)|^q \mathrm{d}x
\le(1+\delta_2)\!\int_U|\nabla^m u|^q \mathrm{d}x
+C_{\delta_2}\,C_\eta^q
\sum_{j=0}^{m-1}\!\int_U|\nabla^j u|^q \mathrm{d}x.   
\end{align}
For each \(0\le j\le m-1\), by applying the Gagliardo--Nirenberg interpolation inequality \cite{Sukochev-Yang-Zanin-2025} and Young's inequality with $\zeta>0$, that is, $st \leq \zeta s^p + C_\zeta t^{p'}$, where $s,t>0$, $p^\prime$ is the conjugate of $p>1$ and $C_\zeta=(\zeta p)^{-\frac{p^\prime}{p}}(p^\prime)^{-1}$, we obtain by setting 
$p = \frac{m}{j}$, $p' = \frac{m}{m-j}$, $s= \|\nabla^m u\|_{L^q}^{\frac{qj}{m}}$, and $t = \|u\|_{L^q}^{{q(1-\frac{j}{m})}}$ that
\[
\|\nabla^j u\|_{L^q}^q
\le\delta_3\|\nabla^m u\|_{L^q}^q
+C_{j,m,\delta_3}\|u\|_{L^q}^q,
\]
where \(\delta_3>0\) is a suitable constant. Hence, from the above inequality, we have
\begin{align}\label{eq1.5}
  \sum_{j=0}^{m-1}\!\int_U|\nabla^j u|^q \mathrm{d}x
\le m\,\delta_3\!\int_U|\nabla^m u|^q \mathrm{d}x
+\Bigl(\sum_{j=0}^{m-1}C_{j,m,\delta_3}\Bigr)\!\int_U|u|^q \mathrm{d}x.  
\end{align}
Set
\[
\delta_4=\delta_2+C_{\delta_2}C_\eta^q\,m\,\delta_3\quad\text{and}\quad
C_{\rm interp}=C_{\delta_2}C_\eta^q\sum_{j=0}^{m-1}C_{j,m,\delta_3}.
\]
Thus, we easily obtain from \eqref{eq1.5} that
\begin{equation}\label{eq3.3.10new}
\int_U|\nabla^m(\eta u)|^q \mathrm{d}x
\le(1+\delta_4)\!\int_U|\nabla^m u|^q \mathrm{d}x
+C_{\rm interp}\!\int_U|u|^q \mathrm{d}x.
\end{equation}
Next, applying Sobolev embedding we have
\begin{equation}\label{eq3.3.12mod}
\Bigg(\!\int_{U\setminus U_\delta}|u|^{q^*}\mathrm{d}x\Bigg)^{\frac{q}{q^*}}
\le\frac1{S_q}\int_U\bigl(|\nabla^m u|^q+\widetilde C_q|u|^q\bigr)\mathrm{d}x.
\end{equation}
where $\widetilde C_q>0$ is a suitable constant. It follows from the hypothesis \eqref{3.1.2} that
\begin{equation}\label{eq3.3.13mod}
\int_U|\nabla^m u|^q \mathrm{d}x
\le\frac1{a_0}\int_\Omega a(x)|\nabla^m u|^q \mathrm{d}x.
\end{equation}
Given \(\epsilon>0\), choose
\(\delta_2,\delta_3\) such that \(\delta_4=\frac{a_0S_q\epsilon}{4}\) 
and \(\delta_1=\frac{a_0S_q \epsilon}{4\widetilde C_q}\). Now, we again set
\[
C_\epsilon
=\frac{4}{S_q}\max\{\,C_{\rm interp},\,\widetilde C_q,\,\delta_1^\frac{q}{q^*}\,\widetilde C_q\}.
\]
Combining \eqref{eq3.3.9mod}, \eqref{eq3.3.10mod}, \eqref{eq3.3.10new}, \eqref{eq3.3.12mod}, and \eqref{eq3.3.13mod} yields \eqref{eq3.3.8}. This ends the proof.
\end{proof}

\section{Proofs of Proposition \ref{propcompactness} and Proposition \ref{propasymptotics}}
This section is devoted to the proof of the results of Propositions \ref{propcompactness} and \ref{propasymptotics}.
To establish it, we derive the following fundamental lemmas concerning the functional \(E\) defined in \eqref{eq3.2.10}.

Let $\{u_j\}_{j\in\mathbb{N}}$ be a bounded $(PS)_\beta$-sequence for the functional $E$.
Then there exists a function $u\in W_0^{m,\mathcal{H}}(\Omega)$ and a subsequence,
still denoted by $\{u_j\}$, such that
\begin{equation}
\begin{cases}
u_j \rightharpoonup u 
& \text{weakly in } W_0^{m,\mathcal{H}}(\Omega), \\
u_j \to u 
& \text{strongly in } L^{\theta}(\Omega) 
\quad \text{for all } \theta\in[1,p^*), \\
u_j \to u 
& \text{a.e. in } \Omega, 
\end{cases}
\end{equation}

Now, by using similar to the proof of the concentration-compactness principle, see \cite[Lemma I.1]{Lions-1985}, we can easily show that 
\begin{equation}
    \begin{cases}
        |\nabla^m u_j|^{p}\overset{*}\rightharpoonup \zeta~\text{in}~\mathcal M(\overline{\Omega})\\
        |u_j|^{p^*}\overset{*}\rightharpoonup\xi~\text{in}~\mathcal M(\overline{\Omega})\\
        a(x)|\nabla^m u_j|^{q}\overset{*}\rightharpoonup \overline\zeta~\text{in}~\mathcal M(\overline{\Omega})\\
        b(x)|u_j|^{q^*}\overset{*}\rightharpoonup\overline\xi~\text{in}~\mathcal M(\overline{\Omega})\\
        
    \end{cases}
\end{equation}
where $\xi,\zeta,\overline\xi,\overline\zeta$ are bounded Radon measures in $\mathcal M(\overline{\Omega})$, set of all Radon measures on $\overline\Omega$. 
 Moreover, the set of concentration points \(\{x_i\}_{i \in J}\) is finite.
 Consider the total energy measure \(\eta = \zeta + \overline{\zeta}\). 

 We define the singular set as
\[
\mathcal{A}_\sigma = \left\{ x \in \overline{\Omega} : \eta(B_r(x) \cap \overline{\Omega}) \ge \sigma \text{ for all } r > 0 \right\}.
\]
We claim that $\mathcal{A}_\sigma$ is a finite set. Suppose, by contradiction, that there exists a sequence of distinct points $(x_s)_{s \in \mathbb{N}}$ in $\mathcal{A}_\sigma$.
By the definition of $\mathcal{A}_\sigma$, for every $x_s$, we have
$ \eta(B_r(x_s) \cap \overline{\Omega}) \ge \sigma$, for all $r>0.$ So we have that $\eta\{x_s\}\geq\sigma$. Since the points are distinct, the disjoint additivity of the measure implies
$\eta(\mathcal{A}_\sigma) = +\infty.$

However, this contradicts the boundedness of the total energy, since the boundedness of \((u_j)\) implies
\[\eta(A_{\sigma}) \le \liminf_{n\to\infty} \int_{A_{\sigma}} (|\nabla^m u_j|^p + a(x)|\nabla^m u_j|^q) \,dx < \infty.\] Thus, the set $\mathcal{A}_\sigma$ must be finite, i.e., $\mathcal{A}_\sigma = \{x_1, x_2, \dots, x_\nu\}$.
\begin{lemma} \label{lemmaweak-convergence}
Let \((u_j)_{j\in\mathbb{N}} \subset W_0^{m,\mathcal{H}}(\Omega)\) be a sequence that weakly converges to \(u\) with \(E'(u_j) \to 0\). Then, along a subsequence, \(\nabla^m u_j \to \nabla^m u\) a.e. in \(\Omega\).
\end{lemma}
\begin{proof}
From the weak convergence of $(u_j)_{j\in\mathbb{N}}$, we have the boundedness of $(u_j)_{j\in\mathbb{N}}$ in $W_0^{m,\mathcal{H}}(\Omega)$. Let $\epsilon_0 > 0$ be fixed and small enough such that $B_{\epsilon_0}(x_s)\cap B_{\epsilon_0}(x_k) = \emptyset$ if $s \ne k$, and let
\[
{{\Omega}}_{\epsilon_0} = \{ x \in {\overline{\Omega}} :~ \|x-x_k\| \ge \epsilon_0,\, k=1,2,\dots,\nu \},
\]
for some $\nu \in \mathbb N$.
Then we claim that
\[
\int_{{\Omega}_{\epsilon_0}} 
\left( |\nabla^m u_j|^{p-2}\nabla^m u_j - |\nabla^m u|^{p-2}\nabla^m u \right)
(\nabla^m u_j - \nabla^m u) \,\mathrm{d}x \to 0.
\]
For this, let $0<\epsilon<\epsilon_0$ and let $\varphi \in C_0^\infty(\mathbb{R}^n, [0,1])$ be such that $\varphi\equiv 1$ in $B_{1/2}(0)$ and $\varphi\equiv 0$ in ${\Omega}\setminus B_1(0)$. Taking 
\[
\psi_\epsilon(x) = 1 - \sum_{k=1}^\nu \varphi\left( \frac{x-x_k}{\epsilon} \right),
\]
we have $0 \le \psi_\epsilon \le 1$, $\psi_\epsilon \equiv 1$ in ${\Omega}_\epsilon = {\Omega} \setminus \bigcup_{k=1}^\nu B(x_k,\epsilon)$, $\psi_\epsilon \equiv 0$ in $\bigcup_{k=1}^\nu B(x_k, \epsilon/2)$ and ${(\psi_\epsilon u_j)_{j\in\mathbb{N}}}$ is a bounded sequence in $W_0^{m, \mathcal{H}}({{\Omega}})$.
Since ${E'(u_j)}\to 0$ so for any $v \in W_0^{m, \mathcal{H}}({{\Omega}})$, we have
\begin{align}\label{weakderivative}
\int_{\Omega} (|\nabla^m u_j|^{p-2}+a(x)|\nabla^m u_j|^{q-2} )\nabla^m u_j .\nabla^mv \,\mathrm{d}x
-\int_{\Omega} (\mu|u_j|^{p^*-2}u_j + b(x)|u_j|^{q^*-2}u_j+&g(x,u_j))v \,\mathrm{d}x\to 0
\end{align}
Putting $v = \psi_\epsilon u_j$, we get
\begin{align*}
\int_{\Omega} (|\nabla^m u_j|^{p-2}+a(x)|\nabla^m u_j|^{q-2} )\nabla^m u_j .\nabla^m(\psi_\epsilon u_j) \,\mathrm{d}x
-\int_{\Omega} (\mu|u_j|^{p^*-2}u_j + b(x)|u_j|^{q^*-2}u_j+&g(x,u_j))\psi_\epsilon u_j\,\mathrm{d}x\\\nonumber
&\le \epsilon_j \| \psi_\epsilon u_j \|,
\end{align*}
where $\epsilon_j>0$ is a constant, and it implies that
\begin{align}
\int_{\Omega}& |\nabla^m u_j|^{p} \psi_\epsilon \,\mathrm{d}x
+ \sum_{l=1}^m \binom{m}{l} \int_{\Omega} |\nabla^m u_j|^{p-2} \nabla^m u_j \nabla^l \psi_\epsilon \nabla^{m-l}u_j \,\mathrm{d}x\nonumber\\
+&\int_{\Omega}a(x) |\nabla^m u_j|^{q} \psi_\epsilon \,\mathrm{d}x
+ \sum_{l=1}^m \binom{m}{l} \int_{\Omega} a(x)|\nabla^m u_j|^{q-2} \nabla^m u_j \nabla^l \psi_\epsilon \nabla^{m-l}u_j \,\mathrm{d}x \label{ppart}\\
&- \int_{\Omega} (\mu|u_j|^{p^*-2}u_j + b(x)|u_j|^{q^*-2}u_j+g(x,u_j))\psi_\epsilon u_j \,\mathrm{d}x \le \epsilon_j \|\psi_\epsilon u_j\|. \nonumber
\end{align}
Now, \eqref{weakderivative} with $v=-\psi_\epsilon u$,
\begin{align*}
\int_{\Omega} (|\nabla^m u_j|^{p-2}+a(x)|\nabla^m u_j|^{q-2} )\nabla^m u_j \nabla^m(-\psi_\epsilon u ) \,\mathrm{d}x
+\int_{\Omega} (\mu|u_j|^{p^*-2}u_j + b(x)|u_j|^{q^*-2}u_j+&g(x,u_j))\psi_\epsilon u \,\mathrm{d}x\nonumber\\
&\le \epsilon_j \| \psi_\epsilon u\|,
\end{align*}
which implies that
\begin{align}
-\int_{\Omega}& |\nabla^m u_j|^{p-2}\nabla^m u_j\nabla^m u \psi_\epsilon \,\mathrm{d}x
- \sum_{l=1}^m \binom{m}{l} \int_{\Omega} |\nabla^m u_j|^{p-2} \nabla^m u_j \nabla^l \psi_\epsilon \nabla^{m-l}u \,\mathrm{d}x\nonumber\\
-&\int_{\Omega}a(x) |\nabla^m u_j|^{q-2}\nabla^m u_j\nabla^m u \psi_\epsilon\, \mathrm{d}x
- \sum_{l=1}^m \binom{m}{l} \int_{\Omega} a(x)|\nabla^m u_j|^{q-2} \nabla^m u_j \nabla^l \psi_\epsilon \nabla^{m-l}u\, \mathrm{d}x\label{qpart}\\
&+ \int_{\Omega} (\mu|u_j|^{p^*-2}u_j + b(x)|u_j|^{q^*-2}u_j+g(x,u_j))\psi_\epsilon u\, \mathrm{d}x \le \epsilon_j \|\psi_\epsilon u\|.\nonumber 
\end{align}
Using the strict convexity of function $t \mapsto |t|^k$, where integer $k>1$, we get
\[
0 \le \left( |\nabla^m u_j|^{p-2}\nabla^m u_j - |\nabla^m u|^{p-2}\nabla^m u \right) (\nabla^m u_j - \nabla^m u),
\]
and consequently
\[
0 \le \int_{{\Omega}_{\epsilon_0}}
    \left( |\nabla^m u_j|^{p-2}\nabla^m u_j - 
           |\nabla^m u|^{p-2}\nabla^m u \right) 
           (\nabla^m u_j - \nabla^m u)\, \mathrm{d}x
\]
\[
\le \int_{\Omega} \left( |\nabla^m u_j|^{p-2}\nabla^m u_j - 
           |\nabla^m u|^{p-2}\nabla^m u \right) 
           (\nabla^m u_j - \nabla^m u)\, \mathrm{d}x. 
\]
This can be written as
\begin{equation}\label{pconvex}
    0 \le \int_{\Omega} \left( |\nabla^m u_j|^{p} \psi_\epsilon 
- |\nabla^m u_j|^{p-2} \psi_\epsilon \nabla^m u_j \nabla^m u 
- |\nabla^m u|^{p-2} \psi_\epsilon \nabla^m u \nabla^m u_j 
+ |\nabla^m u|^{p} \psi_\epsilon \right)\, \mathrm{d}x.
\end{equation}
Similarly,
\begin{equation}\label{qconvex}
0 \le \int_{\Omega} a(x)\left( |\nabla^m u_j|^{q} \psi_\epsilon 
- |\nabla^m u_j|^{q-2} \psi_\epsilon \nabla^m u_j \nabla^m u 
- |\nabla^m u|^{q-2} \psi_\epsilon \nabla^m u \nabla^m u_j 
+ |\nabla^m u|^{q} \psi_\epsilon \right)\, \mathrm{d}x.
\end{equation}
From \eqref{ppart}, \eqref{qpart}, \eqref{pconvex} and \eqref{qconvex}, we obtain
\begin{align}\label{equationestimate}
0 \le &
- \int_{\Omega} |\nabla^m u_j|^{p} \psi_\epsilon\, \mathrm{d}x 
- \sum_{l=1}^m \binom{m}{l} \int_{\Omega} |\nabla^m u_j|^{p-2}\nabla^m u_j \nabla^l \psi_\epsilon \nabla^{m-l}u_j\, \mathrm{d}x \\\nonumber
&-\int_{\Omega}a(x) |\nabla^m u_j|^{q} \psi_\epsilon\, \mathrm{d}x
- \sum_{l=1}^m \binom{m}{l} \int_{\Omega} a(x)|\nabla^m u_j|^{q-2} \nabla^m u_j \nabla^l \psi_\epsilon \nabla^{m-l}u_j\, \mathrm{d}x\\\nonumber
& + \int_{\Omega} (\mu|u_j|^{p^*-2}u_j + b(x)|u_j|^{q^*-2}u_j+g(x,u_j))\psi_\epsilon u_j\, \mathrm{d}x+ \epsilon_j \|\psi_\epsilon u_j\| \\\nonumber
& + \int_{\Omega} |\nabla^m u_j|^{p-2}\nabla^m u_j \nabla^m u \psi_\epsilon\, \mathrm{d}x + \epsilon_j \|\psi_\epsilon u\| \\\nonumber
& + \sum_{l=1}^{m}\binom{m}{l} \int_{\Omega} |\nabla^m u_j|^{p-2}\nabla^m u_j \nabla^l\psi_\epsilon \nabla^{m-l}u\, \mathrm{d}x
 + \int_{\Omega}a(x) |\nabla^m u_j|^{q-2}\nabla^m u_j \nabla^m u \psi_\epsilon\, \mathrm{d}x\\\nonumber
 &+\sum_{l=1}^{m}\binom{m}{l} \int_{\Omega} a(x)|\nabla^m u_j|^{q-2}\nabla^m u_j \nabla^l\psi_\epsilon \nabla^{m-l}u\, \mathrm{d}x- \int_{\Omega} (\mu|u_j|^{p^*-2}u_j + b(x)|u_j|^{q^*-2}u_j+g(x,u_j))\psi_\epsilon u\, \mathrm{d}x\\\nonumber
& + \int_{\Omega} \left( |\nabla^m u_j|^{p}\psi_\epsilon
- |\nabla^m u_j|^{p-2}\psi_\epsilon \nabla^m u_j \nabla^m u
- |\nabla^m u|^{p-2}\psi_\epsilon \nabla^m u\nabla^m u_j
+ |\nabla^m u|^{p}\psi_\epsilon \right)\, \mathrm{d}x\\\nonumber
&+\int_{\Omega} a(x)\left( |\nabla^m u_j|^{q} \psi_\epsilon 
- |\nabla^m u_j|^{q-2} \psi_\epsilon \nabla^m u_j \nabla^m u 
- |\nabla^m u|^{q-2} \psi_\epsilon \nabla^m u \nabla^m u_j 
+ |\nabla^m u|^{q} \psi_\epsilon \right)\, \mathrm{d}x.
\end{align}
Therefore,
\begin{align}
0 \le
& \sum_{l=1}^m \binom{m}{l} \int_{\Omega} |\nabla^m u_j|^{p-2}\nabla^m u_j \nabla^l \psi_\epsilon \nabla^{m-l}(u-u_j)\, \mathrm{d}x
+ \epsilon_j \|\psi_\epsilon u_j\| \nonumber\\
&+\sum_{l=1}^m \binom{m}{l} \int_{\Omega}a(x) |\nabla^m u_j|^{q-2}\nabla^m u_j \nabla^l \psi_\epsilon \nabla^{m-l}(u-u_j)\, \mathrm{d}x+ \epsilon_j \|\psi_\epsilon u\|\label{equationestimate2}\\
& + \int_{\Omega} |\nabla^m u|^{p-2}\nabla^m u (\nabla^m u - \nabla^m u_j)\psi_\epsilon\, \mathrm{d}x
+ \int_{\Omega} a(x)|\nabla^m u|^{q-2}\nabla^m u (\nabla^m u - \nabla^m u_j)\psi_\epsilon\, \mathrm{d}x
\nonumber\\
&+ \int_{\Omega} (\mu|u_j|^{p^*-2}u_j + b(x)|u_j|^{q^*-2}u_j+g(x,u_j))\psi_\epsilon (u_j-u)\, \mathrm{d}x\nonumber.
\end{align}
Now, we estimate each integral in \eqref{equationestimate2} separately. For arbitrary $\delta > 0$, using the interpolation inequality $ab \leq \delta a^{p/(p-1)} + C_\delta b^{p}$ with $C_\delta = \delta^{1-p}$, for all $0\leq r\leq m-1$ and any $l$,
$$
\int_{\Omega} |\nabla^m u_j|^{p-2} \nabla^m u_j  \nabla^l\psi_\epsilon (\nabla^r u - \nabla^r u_j)\, \mathrm{d}x
\leq \delta\int_{\Omega} |\nabla^m u_j|^p\, \mathrm{d}x + C_\delta \int_{\Omega} |\nabla^l\psi_\epsilon|^p |\nabla^r u - \nabla^r u_j|^p\, \mathrm{d}x
$$
$$
\leq \delta K + C_\delta \left( \int_{\Omega} |\nabla^l\psi_\epsilon|^{pt}\, dt \right)^{1/t} \left( \int_{\Omega} |\nabla^r u - \nabla^r u_j|^{ps}\, \mathrm{d}x \right)^{1/s},
$$
where $1/s + 1/t = 1$ and the bound $\int_\Omega |\nabla^m u_{j_k}|^p \mathrm{d}x \leq \rho_{\mathcal{H}}(\nabla^m u_{j_k}) \leq K$ by using \eqref{eq3.3.4}  Thus, since for all $0 \leq r \leq m-1$, {$\nabla^r u_j \to \nabla^r u$ strongly in $L^{sp}({\Omega})$} and $\delta$ is arbitrary in above inequality, we obtain that
\begin{equation}\label{lessthan01}
\limsup_{j\to\infty} \int_{\Omega} |\nabla^m u_j|^{p-2} \nabla^m u_j  \nabla^l\psi_\epsilon (\nabla^r u - \nabla^r u_j)\, \mathrm{d}x \leq 0,
\end{equation}
for all $0 \leq r \leq m-1,$
 and similarly,
 \begin{equation}\label{lessthan02}
\limsup_{j\to\infty} \int_{\Omega} a(x)|\nabla^m u_j|^{q-2} \nabla^m u_j  \nabla^l\psi_\epsilon (\nabla^r u - \nabla^r u_j)\, \mathrm{d}x \leq 0,
\end{equation}
for all $0 \leq r \leq m-1.$
Using that $u_j \rightharpoonup u$ weakly in $W^{m,\mathcal{H}}_0({\Omega})$, we get
\begin{equation}\label{equaltto01}
\int_{\Omega} |\nabla^m u|^{p-2}\nabla^m u (\nabla^{m-l}u - \nabla^{m-l}u_j) \psi_\epsilon\, \mathrm{d}x \to 0 \qquad \text{as } j\to\infty.
\end{equation}
 Similarly, we can deduce 
 \begin{equation}\label{equaltto02}
\int_{\Omega} a(x)|\nabla^m u|^{q-2}\nabla^m u (\nabla^{m-l}u - \nabla^{m-l}u_j) \psi_\epsilon\, \mathrm{d}x \to 0 \qquad \text{as } j\to\infty.
\end{equation}
Now, we claim that
\begin{equation}\label{nonlinearity_limit}
\int_{\Omega} (\mu|u_j|^{p^*-2}u_j + b(x)|u_j|^{q^*-2}u_j + g(x,u_j))\psi_\epsilon (u_j-u)\, \mathrm{d}x \to 0 \quad \text{as } j \to \infty.
\end{equation}
By the definition of the singular set $\{x_1, \dots, x_\nu\}$, we have
\[
u_j \to u \quad \text{strongly in } L^{p^*}_{\text{loc}}({\Omega} \setminus \{x_1, \dots, x_\nu\}).
\]
Since the support of $\psi_\epsilon$ is a compact subset of ${\Omega} \setminus \{x_1, \dots, x_\nu\}$, it follows that
\[
\| u_j - u \|_{L^{p^*}(\text{supp}(\psi_\epsilon))} \to 0 \quad \text{as } j \to \infty.
\]
Applying H\"{o}lder's inequality to the first critical term in \eqref{nonlinearity_limit}, we obtain:
\begin{align}
\left| \int_{\Omega} \mu |u_j|^{p^*-2}u_j \psi_\epsilon (u_j-u) \, \mathrm{d}x \right| 
&\le  \mu \| u_j \|_{L^{p^*}(\Omega)}^{p^*-1} \| u_j - u \|_{L^{p^*}(\text{supp}(\psi_\epsilon))}. \label{est_crit_p}
\end{align}
Similarly, for the second critical term involving $b(x)$, using the weighted H\"{o}lder inequality:
\begin{align}
\left|\int_\Omega b(x) |u_j|^{q^*-2}u_j \psi_\epsilon (u_j-u) \,\mathrm{d}x\right| 
&\leq \left(\int_\Omega b(x)|u_j|^{q^*} \,\mathrm{d}x\right)^{\frac{q^*-1}{q^*}} \left(\int_{\text{supp}(\psi_\epsilon)} b(x)|u_j-u|^{q^*} \,\mathrm{d}x\right)^{\frac{1}{q^*}}. \label{est_crit_q}
\end{align}
By Proposition~\ref{prop3.4.13}, the sequence $(u_j)_{j\in\mathbb{N}}$ is bounded in $L^{p^*}(\Omega)$ and $L_b^{q^*}(\Omega)$. Furthermore, since the support of $\psi_\epsilon$ is disjoint from the singular set, the Concentration-Compactness Principle \cite{Lions-1985} implies that $u_j \to u$ strongly in $L^{p^*}_{\text{loc}}$ and $L^{q^*}_{b, \text{loc}}$ on the set $\Omega \setminus \{x_1, \dots, x_\nu\}$. Consequently,
\[
\| u_j - u \|_{L^{p^*}(\text{supp}(\psi_\epsilon))} \to 0 \quad \text{and} \quad \int_{\text{supp}(\psi_\epsilon)} b(x)|u_j - u|^{q^*} \,\mathrm{d}x \to 0.
\]
Combining this with estimates \eqref{est_crit_p} and \eqref{est_crit_q}, both integrals vanish as $j \to \infty$.

Again applying Hölder's inequality to the last term of \eqref{nonlinearity_limit}, we get
\begin{align}
\int_{\Omega} g(x, u_j) \psi_\epsilon (u_j - u)  \mathrm{d}x 
&\leq \int_{\Omega} \left( c_1 + c_2 |u_j|^{r-1} + c(x)|u_j|^{s-1} \right) \psi_\epsilon|u_j - u|  \mathrm{d}x \notag\\
&\leq c_1 \int_{\Omega} \psi_\epsilon|u_j - u|  \mathrm{d}x \notag \\
&\quad + c_2 \left( \int_{\Omega} \psi_\epsilon^{\frac{r}{r-1}}|u_j|^{r}  \mathrm{d}x \right)^{\frac{r-1}{r}} \left( \int_{\Omega} |u_j - u|^{r}  \mathrm{d}x \right)^{\frac{1}{r}} \nonumber \\
&\quad + \left( \int_{\Omega}\psi_\epsilon^{\frac{s}{s-1}} c(x) |u_j|^{s}  \mathrm{d}x \right)^{\frac{s-1}{s}} \left( \int_{\Omega} c(x) |u_j - u|^{s}  \mathrm{d}x \right)^{\frac{1}{s}}. \label{eq:g_estimate_final}
\end{align}
By Proposition \ref{prop3.3.12}, the embedding $W_0^{m,\mathcal{H}}(\Omega) \hookrightarrow L^{\mathcal{C}}(\Omega)$ is compact. Since $(u_j)_{j\in\mathbb{N}}$ is bounded in $W_0^{m,\mathcal{H}}(\Omega)$, this implies the strong convergence:
\[
u_j \to u \quad \text{strongly in } L^{r}(\Omega) ~\text{and}~u_j \to u \quad \text{strongly in }L^{s}_{c}(\Omega).
\]
Consequently, $\|u_j - u\|_{L^r(\Omega)} \to 0$ and $|u_j - u|_{s,c} \to 0$ as $j \to \infty$. Therefore, every term in \eqref{eq:g_estimate_final} vanishes. So \eqref{nonlinearity_limit} holds.

Hence, from \eqref{equationestimate2}, \eqref{lessthan01}, \eqref{lessthan02}, \eqref{equaltto01}, \eqref{equaltto02} and \eqref{nonlinearity_limit} we have
$$
\limsup_{j\to\infty} \sum_{l=1}^m \binom{m}{l} \int_{\Omega} |\nabla^m u_j|^{p-2} \nabla^m u_j \cdot \nabla^l\psi_\epsilon (\nabla^{m-l}u - \nabla^{m-l}u_j)\, \mathrm{d}x = 0.
$$
and
$$
\limsup_{j\to\infty} \sum_{l=1}^m \binom{m}{l} \int_{\Omega} a(x)|\nabla^m u_j|^{q-2} \nabla^m u_j \cdot \nabla^l\psi_\epsilon (\nabla^{m-l}u - \nabla^{m-l}u_j)\, \mathrm{d}x = 0.
$$
Now, using $\langle E'(u_j), \psi_\epsilon(u_j-u)\rangle \to 0$, we conclude
\begin{align}
\int_{{\Omega}_{\epsilon_0}} 
&\left( |\nabla^m u_j|^{p-2}\nabla^m u_j - |\nabla^m u|^{p-2}\nabla^m u \right)
(\nabla^m u_j - \nabla^m u)\, \mathrm{d}x\\
&+\int_{{\Omega}_{\epsilon_0}} a(x)\left( |\nabla^m u_j|^{q-2}\nabla^m u_j - |\nabla^m u|^{q-2}\nabla^m u \right)
(\nabla^m u_j - \nabla^m u)\, \mathrm{d}x \to 0.
\end{align}
Finally, using this, since $\epsilon_0$ is arbitrary, we obtain that
$$
\nabla^m u_j(x) \to \nabla^m u(x) \quad \text{a.e. in } {\Omega},
$$
which finishes the proof.
\end{proof}
\begin{lemma}\label{lemmaweak-solution}
Let $(u_j)_{j\in\mathbb{N}} \subset W_0^{m,\mathcal{H}}(\Omega)$ satisfy $u_j \rightharpoonup u$ weakly and $E'(u_j) \to 0$. Then $u$ is a weak solution of \eqref{3.2.8}.
\end{lemma}
\begin{proof}
From weak convergence of $(u_j)_{j\in\mathbb{N}}$, we have boundedness of $(u_j)_{j\in\mathbb{N}}$ in $W_0^{m,\mathcal{H}}(\Omega)$.  For all test functions $v \in W_0^{m,\mathcal{H}}(\Omega)$, we have
\begin{align}\label{eq3.4.8}
\int_{\Omega} 
&\Big[ |\nabla^m u_j|^{p-2}\nabla^m u_j \cdot \nabla^m v + a(x)|\nabla^m u_j|^{q-2}\nabla^m u_j \cdot \nabla^m v \notag\\
&- \mu |u_j|^{p^*-2}u_j v - b(x)|u_j|^{q^*-2}u_j v-g(x,u_j)v \Big]\,\mathrm{d}x = o_j(1) 
\end{align}
as $j \to \infty$ so we aim to establish convergence of each term separately in \eqref{eq3.4.8}. Consider an arbitrary subsequence $(u_{j_k})_{j\in\mathbb{N}}$ of $(u_j)_{j\in\mathbb{N}}$ in $W_0^{m,\mathcal{H}}(\Omega)$. By the embedding $W_0^{m,\mathcal{H}}(\Omega) \hookrightarrow W_0^{m,p}(\Omega)$ and boundedness of $(u_j)_{j\in\mathbb{N}}$ in $W_0^{m,\mathcal{H}}(\Omega)$, we derive that $\|\nabla^m u_{j_k}\|_p$ is uniformly bounded and  Lemma \ref{lemmaweak-convergence} provides $\nabla^m u_{j_k} \to \nabla^m u$ a.e. up to a subsequence. Applying \cite[Proposition A.8-(i)]{Autuori-Pucci-2013}  with weight $w = \chi _{\Omega}$ yields
\[
|\nabla^m u_{j_k}|^{p-2}\nabla^m u_{j_k} \rightharpoonup |\nabla^m u|^{p-2}\nabla^m u \quad \text{in } (L^{p'}(\Omega))^{N}\quad \text{when m is odd,}
\]
and
\[
|\nabla^m u_{j_k}|^{p-2}\nabla^m u_{j_k} \rightharpoonup |\nabla^m u|^{p-2}\nabla^m u \quad \text{in } L^{p'}(\Omega)\quad \text{when m is even},
\]
where $p'$ denotes the conjugate of $p$. The subsequence arbitrariness implies the entire sequence \( (|\nabla ^m u_j|^{p-2}\nabla ^m u_j) \) exhibits weak convergence to \( |\nabla ^m u|^{p-2}\nabla ^m u \) within $(L^{p'}(\Omega))^{N}$, when $m$ is odd and  within $L^{p'}(\Omega)$, when $m$ is even . Thus for any $v \in W_0^{m,\mathcal{H}}(\Omega)$
\begin{align}\label{eq3.4.9}
\int_{\Omega} |\nabla^m u_j|^{p-2}\nabla^m u_j \cdot \nabla^m v  \mathrm{d}x \to \int_{\Omega} |\nabla^m u|^{p-2}\nabla^m u \cdot \nabla^m v  \mathrm{d}x\quad\text{as}\quad j\to\infty. 
\end{align}
Since $\nabla^mv\in (L^{\mathcal{H}}(\Omega))^{N}\subset (L^{p}(\Omega))^{N}$, when $m$ is odd and $\nabla^mv\in L^{\mathcal{H}}(\Omega)\subset L^{p}(\Omega)$, when $m$ is even. Similarly, the bound $\int_\Omega a(x)|\nabla^m u_{j_k}|^q \mathrm{d}x \leq \rho_{\mathcal{H}}(\nabla^m u_{j_k}) \leq C$ by using \eqref{eq3.3.4} and a.e. convergence yield, via \cite[Proposition A.8-(i)]{Autuori-Pucci-2013} with weight $w = a(x)\chi_{\Omega}$
\[
|\nabla^m u_{j_k}|^{q-2}\nabla^m u_{j_k} \rightharpoonup |\nabla^m u|^{q-2}\nabla^m u \quad \text{in } (L^{q'}_a(\Omega))^{N} \quad \text{when m is odd, }
\]
and 
\[
|\nabla^m u_{j_k}|^{q-2}\nabla^m u_{j_k} \rightharpoonup |\nabla^m u|^{q-2}\nabla^m u \quad \text{in } L^{q'}_a(\Omega)\quad \text{when m is even, }
\]
where $q'$ denotes the conjugate of $q$.
Consequently, the entire sequence \( (|\nabla ^m u_j|^{q-2}\nabla ^m u_j) \) exhibits weak convergence to \( |\nabla ^m u|^{q-2}\nabla ^m u \) within $(L^{q'}_a(\Omega))^{N}$, when $m$ is odd and within $L^{q'}_a(\Omega)$, when $m$ is even. For arbitrary $v \in W_0^{m,\mathcal{H}}(\Omega)$, we have
\begin{equation}\label{eq3.4.10}
    \int_{\Omega} a(x)|\nabla ^m u_j|^{q-2}\nabla ^m u_j \cdot \nabla ^m v  \mathrm{d}x \to \int_{\Omega} a(x)|\nabla ^m u|^{q-2}\nabla ^m u \cdot \nabla ^m v  \mathrm{d}x \quad \text{as } j \to \infty,
\end{equation}
which follows from the inclusion $\nabla ^m v \in (L^\mathcal{H}(\Omega))^{N} \subset (L^{q}_a(\Omega))^{N}$ when $m$ is odd and $\nabla ^m v \in L^\mathcal{H}(\Omega) \subset L^{q}_a(\Omega)$ when $m$ is even.
Similarly, using the continuous embedding $W_0^{m,\mathcal{H}}(\Omega) \hookrightarrow L^{p^*}(\Omega)$ and the weak convergence (hence boundedness) of $\left(u_j\right)_{j\in\mathbb{N}}$ in $W_0^{m,\mathcal{H}}(\Omega)$, the sequence $(u_{j_k})_{j\in\mathbb{N}}$ remains bounded in $L^{p^*}(\Omega)$. Selecting an appropriate subsequence yields $u_{j_k} \to u$ a.e. in $\Omega$ as $j\to\infty$. Applying reference \cite[Proposition A.8-(i)]{Autuori-Pucci-2013},  with $w = \chi_\Omega$ and noting subsequence arbitrariness, we deduce $|u_j|^{p^*-2}u_j \rightharpoonup|u|^{p^*-2}u$ in $L^{(p^*)'}(\Omega)$. Thus, for each $v \in W_0^{m,\mathcal{H}}(\Omega)$,
\begin{align}\label{eq3.4.11}
    \int_{\Omega} |u_j|^{p^*-2}u_jv  \mathrm{d}x \to \int_{\Omega} |u|^{p^*-2}uv  \mathrm{d}x \quad \text{as } j \to \infty.
\end{align}
Parallel reasoning establishes that for every $v \in W_0^{m,\mathcal{H}}(\Omega)$,
\begin{align}\label{eq3.4.12}
\int_{\Omega} b(x)|u_j|^{q^*-2}u_jv  \mathrm{d}x \to \int_{\Omega} b(x)|u|^{q^*-2}uv  \mathrm{d}x \quad \text{as } j \to \infty.
\end{align}
Finally, combining the continuity of $g(x, \cdot)$ with assumption \eqref{3.2.9}, the $g(x, u_j)$ term converges such that for all $v \in W_0^{m,\mathcal{H}}(\Omega)$,
\begin{align}\label{eq3.4.13}
\int_{\Omega} g(x, u_j)v  \mathrm{d}x \to \int_{\Omega} g(x, u)v  \mathrm{d}x \quad \text{as } j \to \infty.
\end{align}
Combining \cref{eq3.4.9,eq3.4.10,eq3.4.11,eq3.4.12,eq3.4.13},
 we pass the limit in \eqref{eq3.4.8} to achieve the required outcome.
\end{proof}
Now we provide the proofs of Propositions \ref{propcompactness} and \ref{propasymptotics}.
\begin{proof}[\bf{Proof of Proposition \ref{propcompactness}}]
 Consider a (PS)$_\beta$ sequence $(u_j)_{j\in\mathbb{N}}\subset W_0^{m,\mathcal{H}}(\Omega)$ satisfying
\begin{equation}\label{eq3.4.14}
    E(u_j) = \int_{\Omega} \left( \frac{1}{p} |\nabla ^m u_j|^p + \frac{a(x)}{q} |\nabla ^m u_j|^q - \frac{\mu}{p^*} |u_j|^{p^*} - \frac{b(x)}{q^*} |u_j|^{q^*} - G(x, u_j) \right)  \mathrm{d}x = \beta + o_j(1),
\end{equation}
and
\begin{equation}\label{eq3.4.15}
E'(u_j) u_j = \int_{\Omega} \left( |\nabla ^m u_j|^p + a(x)|\nabla ^m u_j|^q - \mu |u_j|^{p^*} - b(x)|u_j|^{q^*} - u_j g(x, u_j) \right) \mathrm{d}x = o_j(\|u_j\|). 
\end{equation}
dividing \eqref{eq3.4.15} by $\sigma$ (defined in \eqref{3.2.11}), subtracting from \eqref{eq3.4.14}, and using \eqref{3.2.11} with $\sigma < q^*$ yields
\[\left( \frac{1}{p} - \frac{1}{\sigma} \right) \int_{\Omega} |\nabla ^m u_j|^p \mathrm{d}x + \left( \frac{1}{q} - \frac{1}{\sigma} \right) \int_{\Omega} a(x)|\nabla ^m u_j|^q \mathrm{d}x \leq o(\|u_j\|) + \beta + c_3 |\Omega| + o_j(1).\]
This inequality combined with \eqref{eq3.3.4} establishes boundedness of $(u_j)_{j\in\mathbb{N}}$. By reflexivity of $W_0^{m,\mathcal{H}}(\Omega)$, up to a subsequence (still denoted by $(u_j)_{j\in\mathbb{N}}$), $(u_j)_{j\in\mathbb{N}}$ converges weakly to some $u \in W_0^{m,\mathcal{H}}(\Omega)$. Lemma \ref{lemmaweak-solution} confirms that $u$ weakly solves \eqref{3.2.8}.
To establish nontriviality, on contrary, we assume $u \equiv 0$. From Proposition \ref{prop3.3.12} and growth condition \eqref{3.2.9}, we already have
\begin{equation}\label{eq3.4.16}
    \int_{\Omega} G(x, u_j) \mathrm{d}x \to 0, \quad\text{and}~~ \int_{\Omega} u_j g(x, u_j) \mathrm{d}x \to 0~~~\text{as}~j\to\infty.
\end{equation}
This reduces \eqref{eq3.4.14} and \eqref{eq3.4.15} to
\begin{equation}\label{eq3.4.17}
\int_{\Omega} \left( \frac{1}{p} |\nabla ^m u_j|^p + \frac{a(x)}{q} |\nabla ^m u_j|^q - \frac{\mu}{p^*} |u_j|^{p^*} - \frac{b(x)}{q^*} |u_j|^{q^*} \right) \mathrm{d}x = \beta + o_j(1) 
\end{equation}
and
\begin{equation}\label{eq3.4.18}
\int_{\Omega} \left( |\nabla ^m u_j|^p + a(x)|\nabla ^m u_j|^q - \mu |u_j|^{p^*} - b(x)|u_j|^{q^*} \right) \mathrm{d}x = o(1), 
\end{equation}
respectively. On the other hand, boundedness of $(u_j)_{j\in\mathbb{N}}$ and properties \eqref{eq3.3.4}, \eqref{eq3.3.6}, \eqref{eq3.3.7} ensure existence of nonnegative limits
\[\int_{\Omega} |\nabla ^m u_j|^p \mathrm{d}x \to X, \quad \int_{\Omega} a(x)|\nabla ^m u_j|^q \mathrm{d}x \to Y, \quad \mu \int_{\Omega} |u_j|^{p^*} \mathrm{d}x \to Z, \quad \int_{\Omega} b(x) |u_j|^{q^*} \mathrm{d}x \to W,\]
for a relabeled subsequence and some $X,Y,Z,W\geq 0$. We recall the definition \eqref{eq3.2.12} of $I$ and use \eqref{eq3.4.17}, which gives
\begin{equation}\label{eq3.4.19}\beta = I(X, Y, Z, W),
\end{equation}
It confirms \eqref{eq3.2.14} holds, while \eqref{eq3.4.18} gives us  second part of \eqref{eq3.2.14}. Now, we recall the estimate 
\begin{equation*}
    \mu \int_{\Omega} |u_j|^{p^*} \mathrm{d}x \leq \frac{\mu}{S_p^{\frac{p^*}{p}}} \left( \int_{\Omega} |\nabla ^m u_j|^p \mathrm{d}x \right)^{\frac{p^*}{p}},
    \end{equation*}
using inequality \eqref{eq3.3.6} which on passage to the limit says that third inequality of \eqref{eq3.2.14} holds for $Z$ and $X$. Furthermore, from \eqref{eq3.3.8}
we derive
\begin{equation}\label{eq3.4.22*}
\int_{\Omega} b(x) |{u_j}|^{q^*}  \mathrm{d}x \leq \frac{\|b\|_{\infty}}{(a_0 S_q)^{\frac{q^*}{q}}} 
\left( (1 + \epsilon) \int_{\Omega} a(x) |{\nabla^m u_j}|^q  \mathrm{d}x + C_\epsilon \int_{U} |{u_j}|^q  \mathrm{d}x \right)^{\frac{q^*}{q}}.
\end{equation}
Since $u_j \rightharpoonup 0$ weakly in $W_0^{m,\mathcal{H}}(\Omega)$ as $j\to\infty$, Proposition \ref{pro3.3.9}-(b) with $r = q < p^*$ implies $u_j \to 0$ strongly in $L^q(U)$. Now passing to the limit in \eqref{eq3.4.22*} as $j \to \infty$, we obtain
\begin{equation}\label{eq3.4.23*}
W \leq \frac{\|b\|_{\infty}}{(a_0 S_q)^{\frac{q^*}{q}}} \left((1 + \epsilon) Y\right)^{\frac{q^*}{q}}.
\end{equation}
Again passing limit $\epsilon \to 0$ in \eqref{eq3.4.23*} confirms the last inequality in \eqref{eq3.2.14} for $W$ and $Y$. Thus $(X, Y, Z, W) \in S(\mu, \|b\|_{\infty})$, and \eqref{eq3.4.19} gives $\beta \geq \beta^* (\mu, \|b\|_{\infty})$, contradicting our initial assumption. Hence, $u$ is a non trivial weak solution of \eqref{3.2.8}.
\end{proof}
\begin{proof}[\bf{Proof of Proposition \ref{propasymptotics}}]
Let $\{(X_j, Y_j, Z_j, W_j)\} \subset S(\mu, \|b\|_{\infty})$ be a minimizing sequence for $\beta^* (\mu, \|b\|_{\infty})$. From  \eqref{eq3.2.14}, we obtain
\begin{equation}\label{eq3.4.22}
\begin{split}
(X_j + Y_j) \left( 1 - \frac{\mu}{S_p^{\frac{p^*}{p}}} X_j^{\frac{p^*}{p} - 1} - \frac{\|b\|_{\infty}}{(a_0 S_q)^{\frac{q^*}{q}}} Y_j^{\frac{q^*}{q} - 1} \right) &\leq X_j + Y_j - \frac{\mu}{S_p^{\frac{p^*}{p}}} X_j^{\frac{p^*}{p}} - \frac{\|b\|_{\infty}}{(a_0 S_q)^{\frac{q^*}{q}}} Y_j^{\frac{q^*}{q}} \\
&\leq X_j + Y_j - Z_j - W_j = 0.
\end{split}
\end{equation}
If $X_j + Y_j = 0$, then \eqref{eq3.2.14} implies $Z_j + W_j = 0$, so all terms vanish and $I(X_j, Y_j, Z_j, W_j) = 0$, contradicting \eqref{eq3.2.14}. Thus $X_j + Y_j > 0$, and from \eqref{eq3.4.22}, we obtain
\begin{equation}\label{eq3.4.23}
\frac{\mu}{S_p^{\frac{N}{N-mp}}} X_j^{mp/(N-mp)} + \frac{\|b\|_{\infty}}{(a_0 S_q)^{N/(N-mq)}} Y_j^{mq/(N-mq)} \geq 1,
\end{equation}
{dividing  first part of \eqref{eq3.2.14}$_j$ by $p^*$ and subtracted from \eqref{eq3.2.12}$_j$ with \eqref{eq3.4.23}} yields
\begin{align}\label{eq3.4.24}
I(\bf{X_j}) &= \frac{m}{N} X_j + \left( \frac{1}{q} - \frac{1}{p^*} \right) Y_j + \left( \frac{1}{p^*} - \frac{1}{q^*} \right) W_j \nonumber \\
&\geq \frac{m}{N}\frac{S_p^{N/mp}}{\mu^{\frac{N-mp}{mp}}} \left( 1 - \frac{\|b\|_{\infty}}{(a_0 S_q)^{N/(N-mq)}} Y_j^{mq/(N-mq)} \right)^{\frac{N-mp}{mp}}, 
\end{align}
where $\mathbf{X}_j = (X_j, Y_j, Z_j, W_j)$ and $p < q < p^* < q^*$. Since $S(\mu, 0) \subset S(\mu, \|b\|_{\infty})$, we have $\beta^* (\mu, \|b\|_{\infty}) \leq \beta^* (\mu, 0)$, so $I(\mathbf{X}_j)$ is uniformly bounded with respect to $\|b\|_{\infty} \geq 0$, for fixed $\mu > 0$. Now, if we take large $j$, then
\[
\left( \frac{1}{q} - \frac{1}{p^*} \right) Y_j \leq I(\mathbf{X}_j) \leq \beta^*(\mu, \|b\|_{\infty}) + 1 \leq \beta^*(\mu, 0) + 1,
\]
establishing uniform boundedness of $Y_j$ with respect to $\|b\|_{\infty}$ which implies that as $\|b\|_\infty \to 0^+$ in \eqref{eq3.4.24}, inequality \eqref{eq3.2.18} follows.
Similarly, dividing first part of \eqref{eq3.2.14}$_j$ by $q^*$ and subtracted from \eqref{eq3.2.12}$_j$ with \eqref{eq3.4.23} and \eqref{eq3.2.14} gives
\begin{align}\label{eq3.4.25}
I(\mathbf{X}_j) &= \left( \frac{1}{p} - \frac{1}{q^*} \right) X_j + \frac{m}{N} Y_j - \left( \frac{1}{p^*} - \frac{1}{q^*} \right) Z_j \nonumber \\
&\geq \frac{m}{N} \frac{(a_0 S_q)^{\frac{N}{mq}}}{\|b\|_{\infty}^{\frac{N-mq}{mq}}} \left( 1 - \frac{\mu}{S_p^{\frac{N}{N-mp}}} X_j^{\frac{mp}{(N-mp)}} \right)^{\frac{(N-mq)}{mq}} - \left( \frac{1}{p^*} - \frac{1}{q^*} \right) \frac{\mu}{S_p^{\frac{p^*}{p}}} X_j^{\frac{p^*}{p}}. 
\end{align}
Since $S(0, \|b\|_{\infty}) \subset S(\mu, \|b\|_{\infty})$,  we have $\beta^* (\mu, \|b\|_{\infty}) \leq \beta^* (0, \|b\|_{\infty})$, so $I(\mathbf{X}_j)$ is uniformly bounded in $\mu \geq 0$ for fixed $\|b\|_{\infty} > 0$, so the first equality in \eqref{eq3.4.25} implies uniform boundedness of $X_j$ in $\mu$ which implies that as $\mu \to 0^+$ in \eqref{eq3.4.25}, inequality \eqref{eq3.2.19} follows. This completes the proof.
\end{proof}
\section{Proof of existence results }
This section presents the proof of our main results, outlining the conditions under which our main problem possesses a weak solution. We begin by performing analysis to establish the proof of Theorem \ref{thm3.2.1}.
The variational functional for \eqref{3.2.1} is
\begin{equation}\label{jfunctional}
J(u) = \int_{\Omega} \left( \frac{1}{p} |{\nabla ^m u}|^p + \frac{a(x)}{q} |{\nabla ^m u}|^q - \frac{\lambda}{r} |{u}|^r - \frac{\mu}{p^*} |{u}|^{p^*} - \frac{b(x)}{q^*} |{u}|^{q^*} \right) \mathrm{d}x,
\end{equation}
for $u \in W_0^{m,\mathcal{H}}(\Omega)$. Based on values of $\lambda$, the origin constitutes a strict local minimizer of the functional $J$ for the following reasons:\\
\textbf{Case(I)}: When $r = p$, inequalities \eqref{eq3.2.2star}, \eqref{eq3.3.6}, and \eqref{eq3.3.7} yield
\begin{align}\label{eq3.5.2}
J(u) &\geq \frac{1}{p} \left( 1 - \frac{\lambda}{\lambda_1(p)} \right) \int_{\Omega} |{\nabla ^m u}|^p  \mathrm{d}x + \frac{1}{q} \int_{\Omega} a(x) |{\nabla ^m u}|^q  \mathrm{d}x \nonumber \\
&\quad - \frac{\mu}{p^* S_p^{\frac{p^*}{p}}} \left( \int_{\Omega} |{\nabla ^m u}|^p  \mathrm{d}x \right)^{\frac{p^*}{p}} - \frac{\kappa}{q^*} \left( \int_{\Omega} a(x) |{\nabla ^m u}|^q  \mathrm{d}x \right)^{\frac{q^*}{q}}. 
\end{align}
For $\|u\| \leq 1$, inequality \eqref{eq3.3.4} implies
\begin{equation*}
\|u\|^q \leq \int_{\Omega} \left( |{\nabla ^m u}|^p + a(x) |{\nabla ^m u}|^q \right) \mathrm{d}x \leq \|u\|^p.
\end{equation*}
With $0 < \lambda < \lambda_1(p)$, equation \eqref{eq3.5.2} becomes
\begin{equation*}
J(u) \geq \min \left\{ \frac{1}{p} \left( 1 - \frac{\lambda}{\lambda_1(p)} \right), \frac{1}{q} \right\} \|u\|^q - \frac{\mu}{p^* S_p^{\frac{p^*}{p}}} \|u\|^{p^*} - \frac{\kappa}{q^*} \|u\|^{p \frac{q^*}{q}}.
\end{equation*}
Since $q < p^* <  \frac{pq^*}{q}$, the origin is a strict local minimizer of $J$, when $0 < \lambda < \lambda_1(p)$.\\
\textbf{Case(II)}: When $r > p$,
similar arguments using the embedding $W_0^{m,\mathcal{H}}(\Omega) \hookrightarrow L^r(\Omega)$ (Proposition \ref{pro3.3.9}-(b)) establish that the origin constitutes a strict local minimizer of the functional $J$ for all $\lambda > 0$. 

On the other hand, for any nonzero $u \in W^{m,\mathcal{H}}_0(\Omega)$, we observe $J(tu) \to -\infty$ as $t \to +\infty$. This establishes the mountain pass geometry for $J$, and we define the mountain pass level as
\begin{equation*}\label{eqbeta}
\beta = \inf_{\gamma \in \Gamma} \max_{u \in \gamma([0,1])} J(u) > 0,
\end{equation*}
where the path space $\Gamma$ consists of continuous curves connecting the origin to the negative energy set
\begin{equation*}\label{eqGamma}
\Gamma = \left\{ \gamma \in C([0,1], W^{m,\mathcal{H}}_0(\Omega)) : \gamma(0) = 0, \, J(\gamma(1)) < 0 \right\}.
\end{equation*}
By the mountain pass geometry and the Deformation Lemma, $J$ admits a Palais-Smale sequence $(u_j)_{j\in\mathbb{N}}$ at level $\beta$, see Ambrosetti-Rabinowitz \cite{Ambrosetti-Rabinowitz-1973}, Willem \cite{Willem-1996}. Then the claim below, along with Proposition \ref{propcompactness}, ensures that $(u_j)_{j\in\mathbb{N}}$ contains a subsequence converging weakly to a nontrivial weak solution of problem \eqref{3.2.1}. \\
\textbf{Claim:} For all $\mu > 0$ and sufficiently small $\|b\|_{\infty} \geq 0$, under the cases mentioned in  Theorem \ref{thm3.2.1}, the strict inequality given below holds
\begin{equation}\label{eqbeta_inequality}
\beta < \beta^*(\mu, \|b\|_{\infty}).
\end{equation}
For any nonzero $u \in W^{m,\mathcal{H}}_0(\Omega)$, there exists $t_u > 0$ satisfying $J(t_u u) < 0$, since $\lim_{{t \to \infty}} J(tu) = -\infty$. So the line segment$ \{t_utu:0\leq t\leq 1\}$ belongs to $\Gamma$, yielding the bound
\begin{equation*}\label{eqbeta_bound}
\beta \leq \max_{0 \leq t \leq 1} J(t t_u u) \leq \sup_{t \geq 0} J(tu).
\end{equation*}
Thus, to establish \eqref{eqbeta_inequality}, it suffices to show that for any $u_0 \in W^{m,\mathcal{H}}_0(\Omega) \setminus \{0\}$, 
\begin{equation}\label{eqmax_condition}
\max_{t \geq 0} J(tu_0) < \beta^*(\mu, \|b\|_{\infty}).
\end{equation}
To construct such a function $u_0 $, we fix $x_0 = 0$, $\sigma>0$ and consider a cut-off function $\psi \in C_0^\infty(B_\sigma(0))$ with $0 \leq \psi \leq 1$ and $\psi \equiv 1$ on $B_{\sigma/2}(0)$.
Let
\begin{equation*}
    w_{\epsilon}(x)=\frac{1}{(\epsilon^{\frac{p}{p-1}}+|x|^{\frac{p}{p-1}})^{\frac{N-mp}{p}}}.
\end{equation*}
For any $j=0,1,\cdots,m$, it holds 
\begin{equation*}
|\Delta^jw_\epsilon(x)|\leq C\frac{|x|^{2j}}{(\epsilon^{\frac{p}{p-1}}+|x|^{\frac{p}{p-1}})^{\frac{(N-mp+2jp)}{p}}},
\end{equation*}
\begin{equation*}
|\nabla\Delta^jw_\epsilon(x)|\leq C\frac{|x|^{2j}}{(\epsilon^{\frac{p}{p-1}}+|x|^{\frac{p}{p-1}})^{\frac{(N-mp+2jp+p)}{p}}},
\end{equation*}
where $C>0$ is a constant. For $\epsilon > 0$, we define the following sequence of functions
\begin{align}\label{eqtest_functions}
u_\epsilon(x) &= \frac{\psi(x)}{\left( \epsilon^{\frac{p}{p-1}} + |x|^{\frac{p}{p-1}} \right)^{\frac{N-mp}{p}}} = \psi(x) w_\epsilon, \\
v_\epsilon(x) &= \frac{u_\epsilon(x)}{\|{u_\epsilon}\|_{p^*}}.
\end{align}
Recall that 
$g(\epsilon)=O(f(\epsilon))$
as $\epsilon\to0$ if there exist constants $c_1$, $c_2>0$ such that
$$c_1|f(\epsilon)|\leq|g(\epsilon)|\leq c_2|f(\epsilon)|$$
 for all sufficiently small $\epsilon>0$.
We shall verify \eqref{eqmax_condition} for $u_0 = v_\epsilon$ with sufficiently small $\epsilon > 0$. Following Grunau \cite{Grunau-1995} and Dr{\'a}bek-Huang \cite{Drabek-Huang-1997}, we have the following estimates as $\epsilon \to 0$
\begin{align}
\int_{\Omega} |\nabla^m v_\epsilon|^p  \mathrm{d}x &= S_p + O\left(\epsilon^{\frac{N-mp}{p-1}}\right), \label{eqgrad_est} 
\end{align}
\begin{align}
\int_{\Omega} |\nabla^mv_\epsilon|^q \mathrm{d}x &= 
\begin{cases} 
O \left( \epsilon^{\frac{N(p-q)}{p}} \right) & q > \frac{N(p-1)}{N-m}, \\
O \left( \epsilon^{\frac{N(N-mp)}{(N-m)p}} | \log \epsilon | \right) & q = \frac{N(p-1)}{N-m}, \\
O \left( \epsilon^{\frac{(N-mp)q}{p(p-1)}} \right) & q < \frac{N(p-1)}{N-m},
\end{cases} \label{eqnablam_est}
\end{align}
and
\begin{align}
\int_{\Omega} v_\epsilon^r  \mathrm{d}x &= 
\begin{cases} 
O \left( \epsilon^{\frac{Np-(N-mp)r}{p}} \right) & r > \frac{N(p-1)}{N-mp}, \\
O \left( \epsilon^{\frac{N}{p}} | \log \epsilon | \right) & r = \frac{N(p-1)}{N-mp}, \\
O \left( \epsilon^{\frac{(N-mp)r}{p(p-1)}} \right) & r < \frac{N(p-1)}{N-mp},
\end{cases} \label{eqv_r_est}
\end{align}
\begin{lemma}\label{lemkey_lemma}
{Under the assumptions \eqref{3.1.2} and \eqref{3.2.2}}, if the limit condition
\begin{equation}\label{eqlimit_condition}
\lim_{\epsilon \to 0} \frac{\epsilon^{\frac{N-mp}{p-1}}}{\int_{\Omega} v^r_\epsilon \, \mathrm{d}x} = 0
\end{equation}
holds, then there exist constants $\epsilon_0, b^* > 0$ such that for all $\mu > 0$ and $\|b\|_{\infty} \in [0, b^*)$,
\begin{equation*}\label{eqlemma_conclusion}
\max_{t \geq 0} J(tv_{\epsilon_0}) < \beta^*(\mu, \|b\|_{\infty}).
\end{equation*}
\end{lemma}
\begin{proof}
{Assumption \eqref{3.1.2} and \eqref{3.2.2} imply $a = b = 0$ in $B_\sigma(0)$ because supp(b)$\subset$ supp(a)}. Since $\operatorname{supp}(v_\epsilon) \subset B_\sigma(0)$ and $\|{v_\epsilon}\|_{p^*} = 1$, the energy functional simplifies to
\begin{equation*}\label{eqphi_epsilon}
J(tv_\epsilon) = \frac{t^p}{p} \int_{\Omega} |\nabla^m v_\epsilon|^p \, \mathrm{d}x - \frac{\lambda t^r}{r} \int_{\Omega} v^r_\epsilon \, \mathrm{d}x - \frac{\mu t^{p^*}}{p^*} = \varphi_\epsilon(t).
\end{equation*}
In view of \eqref{eq3.2.18}, it suffices to show
\begin{equation*}\label{eqsup_inequality}
\sup_{t \geq 0} \varphi_\epsilon(t) < \frac{m}{N} \frac{S_p^{\frac{N}{mp}}}{\mu^{\frac{N-mp}{mp}}},
\end{equation*}
for sufficiently small $\epsilon > 0$.
Suppose to the contrary that there exist sequences $\epsilon_j \to 0$ and $t_j > 0$ satisfying $\varphi_{\epsilon_j}^\prime(t_j)=0$ so that we have
\begin{align}
\varphi_{\epsilon_j}(t_j)     &= \frac{t_j^p}{p} \int_{\Omega} |\nabla^m v_j|^p \, \mathrm{d}x  - \frac{\lambda t_j^r}{r} \int_{\Omega} v_j^r \, \mathrm{d}x - \frac{\mu t_j^{p^*}}{p^*} \geq \frac{m}{N} \frac{S_p^{\frac{N}{mp}}}{\mu^{\frac{N-mp}{mp}}}, \label{eqcontrary1}\\
t_j \varphi_{\epsilon_j}^\prime(t_j) &= t_j^p \int_{\Omega} |\nabla^m v_j|^p \, \mathrm{d}x - \lambda t_j^r \int_{\Omega} v_j^r \, \mathrm{d}x - \mu t_j^{p^*} = 0, \label{eqcontrary2}
\end{align}
where $v_j = v_{\epsilon_j}$.
Relations \eqref{eqgrad_est} and \eqref{eqv_r_est} implies  
\begin{align*}
\int_{\Omega} |{\nabla^m v_j}|^p \, \mathrm{d}x \to S_p, ~~
\int_{\Omega} {v_j^r} \, \mathrm{d}x &\to 0\quad\text{as }j\to \infty. 
\end{align*}
Equation \eqref{eqcontrary1} implies the boundedness of $(t_j)$, which converges to some $t_0 > 0$ say, along a subsequence. Passing the limit in \eqref{eqcontrary2} yields
\begin{equation}\label{eqt0_relation}
S_p t_0^p - \mu t_0^{p^*} = 0,
\end{equation}
so $t_0 = \left(\frac{S_p}{\mu}\right)^{\frac{1}{p^*-p}}$.
Combining \eqref{eqcontrary2}, \eqref{eqt0_relation}, and \eqref{eqgrad_est} gives
\begin{equation*}\label{eqresidual_equation1}
S_p(t_j^p - t_0^p) - \lambda t_j^r \int_{\Omega} {v_j}^r \, \mathrm{d}x - \mu(t_j^{p^*} - t_0^{p^*}) = O(\epsilon_j^{\frac{N-mp}{p-1}}).
\end{equation*}
Applying the mean value theorem to the terms $(t_j^p - t_0^p)$ and $(t_j^{p^*} - t_0^{p^*})$
\begin{equation}\label{eqmvt_result}
\left( pS_p\sigma_j^{p-1} - p^*\mu\tau_j^{p^*-1} \right)(t_j - t_0) = \lambda t_j^r \int_{\Omega} {v_j}^r \, \mathrm{d}x + O(\epsilon_j^{\frac{N-mp}{p-1}}),
\end{equation}
where $\sigma_j, \tau_j$ are intermediate values between $t_0$ and $t_j$.
As $t_j \to t_0$, we have $\sigma_j, \tau_j \to t_0$, leading to
\begin{equation*}
pS_p\sigma_j^{p-1} - p^*\mu\tau_j^{p^*-1} \to pS_p t_0^{p-1} - p^*\mu t_0^{p^*-1} = -(p^* - p)\mu t_0^{p^*-1} < 0,
\end{equation*}
by \eqref{eqt0_relation}. Thus, \eqref{eqmvt_result} and \eqref{eqlimit_condition} imply $t_j \leq t_0$ for sufficiently large $j$.
Dividing \eqref{eqcontrary2} by $p^*$, subtracting from \eqref{eqcontrary1}, and using \eqref{eqgrad_est} and \eqref{eqt0_relation}, we obtain
\begin{equation*}\label{eqfinal_inequality}
{\frac{m}{N} S_p t_j^p - \lambda \left( \frac{1}{r} - \frac{1}{p^*} \right) t_j^r \int_{\Omega} {v_j}^r \, \mathrm{d}x \geq \frac{m}{N} S_p t_0^p + O(\epsilon_j^{\frac{N-mp}{p-1}}).}
\end{equation*}
Given that $t_j \leq t_0$ and applying \eqref{eqlimit_condition}, we derive
\begin{equation*}
\lambda \left( \frac{1}{r} - \frac{1}{p^*} \right) t_0^r \leq 0.
\end{equation*}
This is a contradiction since $\lambda, t_0 > 0$ and $r < p^*$ (which implies $\frac{1}{r} - \frac{1}{p^*} > 0$).
Therefore, our assumption was false, and the lemma holds.
\end{proof}
Now, we can finish the proof of Theorem \ref{thm3.2.1}.
\begin{proof}[\bf{Proof of Theorem \ref{thm3.2.1}}]
To complete the proof of Theorem \ref{thm3.2.1}, we only need to justify that
\[
\lim_{\epsilon\to 0}\frac{\epsilon^{\frac{N-mp}{p-1}}}{\int_{\Omega}v_\epsilon^r\,\mathrm{d}x}=0,
\]
in each parameter regime given in statement of Theorem \ref{thm3.2.1}, thanks to lemma \ref{lemkey_lemma}.  From estimate \eqref{eqv_r_est},
\begin{equation}\label{eqasymptotic_ratio}
\frac{\epsilon^{\frac{N-mp}{p-1}}}{\int_{\Omega}v_\epsilon^r\,\mathrm{d}x}
=
\begin{cases}
O\left(\epsilon^{\tfrac{(N-mp)(p-1)-(Np-2N+mp)p}{p(p-1)}}\right) 
& r>\tfrac{N(p-1)}{N-mp},\\[1ex]
O\left(\epsilon^{\tfrac{N-mp^2}{p(p-1)}}/\lvert\log\epsilon\rvert\right) 
& r=\tfrac{N(p-1)}{N-mp},\\[1ex]
O\left(\epsilon^{\tfrac{(N-mp)(p-r)}{p(p-1)}}\right) 
& r<\tfrac{N(p-1)}{N-mp}.
\end{cases}
\end{equation}

\textbf{Case (a)}: \(\,N\ge mp^2,\,r=p\) with $0 < \lambda < \lambda_1(p)$. 
In this case, \eqref{eqasymptotic_ratio} yields
\[
\frac{\epsilon^{\frac{N-mp}{p-1}}}{\int_{\Omega}v_\epsilon^p\,\mathrm{d}x}
=
\begin{cases}
O\left(\epsilon^{\frac{N-mp^2}{p-1}}\right) & N>mp^2,\\[0.5ex]
O\left(\frac{1}{\log\epsilon}\right) & N=mp^2,
\end{cases}
\]
which vanishes as \(\epsilon\to0\) satisfying \eqref{eqlimit_condition}.

\textbf{Case (b)}: \(\,N\ge mp^2,\,p<r<p^*\). Our assumption implies \(p\ge\tfrac{N(p-1)}{N-mp}\) which immediately gives \(r>\tfrac{N(p-1)}{N-mp}\) and so \eqref{eqasymptotic_ratio} yields
\[
\frac{\epsilon^{\frac{N-mp}{p-1}}}{\int_{\Omega}v_\epsilon^r\,\mathrm{d}x}
=O\left(\epsilon^{\tfrac{(N-mp)(p-1)r-(Np-2N+mp)p}{p(p-1)}}\right).
\]
Noting \((N-mp)(p-1)r-(Np-2N+mp)p\ge (N-mp^2)p\ge0\), the ratio tends to zero satisfying \eqref{eqlimit_condition}.

\textbf{Case (c)}: \(\,N<mp^2,\;\tfrac{(Np-2N+mp)p}{(N-mp)(p-1)}<r<p^*\).  
Here again we have, \(r>\tfrac{N(p-1)}{N-mp}\), so \eqref{eqasymptotic_ratio} yields
\[
\frac{\epsilon^{\frac{N-mp}{p-1}}}{\int_{\Omega}v_\epsilon^r\,\mathrm{d}x}
=O\left(\epsilon^{\tfrac{(N-mp)(p-1)r-(Np-2N+mp)p}{p(p-1)}}\right).
\]
We observe that the exponent ${(N-mp)(p-1)r - (Np-2N+mp)p}>0$ by hypothesis on r, satisfying \eqref{eqlimit_condition}. 

Hence, in every case, the limit vanishes, which completes the proof.
\end{proof}

Next we shall establish the proof of Theorem \ref{thm3.2.4}, for which we consider the variational functional for problem \eqref{3.2.5} as
\begin{equation*}\label{eqdouble_phase_func}
F(u) = \int_{\Omega} \left( \frac{1}{p} |{\nabla^m u}|^p + \frac{a(x)}{q} |{\nabla^m u}|^q - \frac{c(x)}{s} |{u}|^s - \frac{\mu}{p^*} |{u}|^{p^*} - \frac{b(x)}{q^*} |{u}|^{q^*} \right) \mathrm{d}x,
\end{equation*}
for $u \in W^{m,\mathcal{H}}_0(\Omega)$. We fix $p < r < p^*$, $p^* \leq s < q^*$ and define
\begin{equation*}\label{eqorlicz_function}
\mathcal{C}(x,t) = t^r + c(x)t^s, \quad (x,t) \in \Omega \times [0,\infty).
\end{equation*}
For $u \in L^{\mathcal{C}}(\Omega)$ with $\|{u}\|_\mathcal{C} \leq 1$, we have 
\begin{equation*}\label{eqestimate}
\int_{\Omega} c(x) |u|^s \, \mathrm{d}x \leq 
\int_{\Omega} \left(|u|^r + c(x)|u|^s \right) \mathrm{d}x \leq 
\max \left\{ \|{u}\|_{\mathcal{C}}^r, \|{u}\|_{\mathcal{C}}^s \right\}= \|{u}\|_{\mathcal{C}}^r.
\end{equation*}
By Proposition \ref{prop3.3.12}, $W^{m,\mathcal{H}}_0(\Omega) \hookrightarrow L^{\mathcal{C}}(\Omega)$ continuously. Analogous to functional \eqref{jfunctional}, $F$ exhibits mountain pass geometry. To establish existence, we aim to show that for any $\|b\|_{\infty} > 0$ and sufficiently small $\mu \geq 0$,
\begin{equation}\label{eq3.5.25}
\max_{t\geq 0} F(tu_0) < \beta^*(\mu, \|b\|_\infty) 
\end{equation}
for some non-zero $u_0 \in W^{m,\mathcal{H}}_0(\Omega)$, where $\beta^*$ is the critical energy level. 
Fix $x_0 = 0$ and consider a cut-off function $\psi \in C^\infty_0(B_\sigma(0))$ with $0 \leq \psi \leq 1$ and $\psi = 1$ on $B_{\sigma/2}(0)$. For parameters $\epsilon > 0$ and $0 < \delta \leq 1$, we define the functions
\begin{align}\label{eqtest_function}
u_{\epsilon,\delta}(x) &= \frac{\psi(\frac{x}{\delta})}{\left( \epsilon^{\frac{q}{q-1}} + |x|^{\frac{q}{q-1}} \right)^{\frac{N-mq}{q}}}, \\
v_{\epsilon,\delta}(x) &= \frac{u_{\epsilon,\delta}(x)}{\|{u_{\epsilon,\delta}}\|_{q^*}}.
\end{align}
Then $\|v_{\epsilon,\delta}\|_{q^*}=1$ and we will derive the estimate similar to \eqref{eqgrad_est}, \eqref{eqnablam_est}, and \eqref{eqv_r_est} for $v_{\epsilon,\delta}$.
First, we note that 
\begin{equation}\label{urelation}
u_{\epsilon,\delta}(x)=\delta^{-\frac{(N-mq)}{q-1}}u_{\frac{\epsilon}{\delta}}\left(\frac{x}{\delta}\right),
\end{equation}
so 
\begin{align*}
\int_{\Omega}u_{\epsilon,\delta}^{q^*}(x)\mathrm{d}x = \delta^{-\frac{Nq}{q-1}}\int_{\Omega}{u_{\frac{\epsilon}{\delta}}}^{q^*}\left(\frac{x}{\delta}\right)\mathrm{d}x
=\delta^{-\frac{N}{q-1}}\int_{\Omega}{u_{\frac{\epsilon}{\delta}}}^{q^*}(x)\mathrm{d}x.
\end{align*}
and hence, 
\begin{equation}\label{urelation2}\|u_{\epsilon,\delta}\|_{q^*}=\delta^{-\frac{N-mq}{q(q-1)}}\|u_{\frac{\epsilon}{\delta}}\|_{q^*}.
\end{equation}
It follows from \eqref{urelation} and \eqref{urelation2} that
$$v_{\epsilon,\delta}(x)=\delta^{-\frac{(N-mq)}{q}}v_{\frac{\epsilon}{\delta}}\left(\frac{x}{\delta}\right),$$
so
\begin{align}\label{vsrelation}
\int_{\Omega}v_{\epsilon,\delta}^s(x)\mathrm{d}{x}=\delta^{-\frac{(N-mq)s}{q}}\int_{\Omega}v_{\frac{\epsilon}{\delta}}^s\left(\frac{x}{\delta}\right)\mathrm{d}x=\delta^{\frac{Nq-(N-mq)s}{q}}\int_{\Omega}v_{\frac{\epsilon}{\delta}}^s(x)\mathrm{d}x.
\end{align}
Moreover
$$\nabla^mv_{\epsilon,\delta}(x)=\delta^{-\frac{N}{q}}\nabla^mv_{\frac{\epsilon}{\delta}}\left(\frac{x}{\delta}\right),$$
and hence 
\begin{align}\label{gradmrelation}
\int_{\Omega}|\nabla^mv_{\epsilon,\delta}|^p\mathrm{d}x=\delta^{-\frac{Np}{q}}\int_{\Omega}|\nabla^mv_{\frac{\epsilon}{\delta}}\left(\frac{x}{\delta}\right)|^p\mathrm{d}{x}=\delta^{\frac{N(q-p)}{q}}\int_{\Omega}|\nabla^mv_{\frac{\epsilon}{\delta}}\left(\frac{x}{\delta}\right)|^p\mathrm{d}{x}.
\end{align}
We will demonstrate that $u_0 = v_{\epsilon,\delta}$ satisfies \eqref{eq3.5.25} for appropriately chosen $\epsilon, \delta > 0$. Combining \eqref{vsrelation} \eqref{gradmrelation} with \eqref{eqgrad_est}, \eqref{eqnablam_est}, \eqref{eqv_r_est},
the following estimates hold as $\epsilon \to 0$ and $\epsilon / \delta \to 0$, referring  Ho et al. \cite[Lemma 3.2]{Ho-Perera-Sim-2023}
\begin{align}
\int_{\Omega} |{\nabla^m v_{\epsilon,\delta}}|^q \, \mathrm{d}x &= S_q + O\left( \left(\frac{\epsilon}{ \delta}\right)^{\frac{N-mq}{q-1}} \right),\label{eq3.5.28} \\
\int_{\Omega} |{\nabla^m v_{\epsilon,\delta}}|^p \, \mathrm{d}x &= 
\begin{cases}
O\left( \epsilon^{\frac{N(q-p)}{q}} \right), & \text{if } p > \frac{N(q-1)}{N-m}, \\
O\left( \epsilon^{\frac{N(N-mq)}{(N-m)q}} |{\log (\epsilon / \delta)}| \right), & \text{if } p = \frac{N(q-1)}{N-m}, \\
O\left( \epsilon^{\frac{(N-mq)p}{q(q-1)}} \delta^{\frac{(N(q-1)-(N-m)p)}{q-1}} \right), & \text{if } p < \frac{N(q-1)}{N-m},
\end{cases} \label{eq3.5.29}
\end{align}
and
\begin{align}
\int_{\Omega} {v^s_{\epsilon,\delta}} \, \mathrm{d}x &=
\begin{cases}
O\left( \epsilon^{\frac{Nq-(N-mq)s}{q}} \right), & \text{if } s > \frac{N(q-1)}{N-mq}, \\
O\left( \epsilon^{\frac{N}{q}} |{\log (\frac{\epsilon}{ \delta})}| \right), & \text{if } s = \frac{N(q-1)}{N-mq}, \\
O\left( \epsilon^{\frac{(N-mq)s}{q(q-1)}} \delta^{\frac{(N(q-1)-(N-mq)s)}{q-1}} \right), & \text{if } s < \frac{N(q-1)}{N-mq}.
\end{cases} \label{eq3.5.30}
\end{align}
\begin{lemma}\label{lemcritical_energy_bound}
Let $v_{\epsilon,\delta}$ be a family of functions satisfying the normalization condition $\|{v_{\epsilon,\delta}}\|_{q^*} = 1$ and $\epsilon/\delta\to0$. Then under the asymptotic conditions
\begin{align}
\lim_{\epsilon \to 0} \frac{\int_{\Omega} |{\nabla^m v_{\epsilon,\delta}}|^p \, \mathrm{d}x}{\int_{\Omega} {v_{\epsilon,\delta}}^s \, \mathrm{d}x} &= 0, \label{eqlimit_condition_1}\\
\lim_{\epsilon \to 0} \frac{\left(\frac{\epsilon}{\delta}\right)^{\frac{N-mq}{q-1}}}{\int_{\Omega} {v_{\epsilon,\delta}}^s \, \mathrm{d}x} &= 0, \label{eqlimit_condition_2}
\end{align}
there exist positive constants $\epsilon_0$, $\delta_0$, and $\mu^*$ such that for all $0 \leq \mu < \mu^*$ and $\|b\|_\infty > 0$
\begin{equation*}
\max_{t\geq 0} F(tv_{\epsilon_0, \delta_0}) < \beta^\star(\mu, \|b\|_\infty).
\end{equation*}
\end{lemma}
\begin{proof}
We proceed by contradiction, so suppose the conclusion fails to hold. Then for any choice of $\epsilon_0$, $\delta_0$, and $\mu^*$, there exists some $\mu \in [0, \mu^*)$ such that
\begin{equation*}
\max_{t\geq 0} F(tv_{\epsilon_0, \delta_0}) \geq \beta^\star(\mu, \|b\|_\infty).
\end{equation*}
{Since $\operatorname{supp}(v_{\epsilon,\delta}) \subset B_{\sigma{\delta}}(0)\subset B_\sigma(0)$} because $0 < \delta \leq 1$ and $\|{v_{\epsilon,\delta}}\|_{q^*} = 1$, the energy functional satisfies
\begin{equation*}\label{eqenergy_upper_bound}
F(tv_{\epsilon,\delta}) \leq \frac{t^p}{p} \int_{\Omega} |{\nabla^m v_{\epsilon,\delta}}|^p \, \mathrm{d}x + \frac{a_0 t^q}{q} \int_{\Omega} |{\nabla^m v_{\epsilon,\delta}}|^q \, \mathrm{d}x - \frac{c_0 t^s}{s} \int_{\Omega} {v_{\epsilon,\delta}}^s \, \mathrm{d}x - \frac{\|b\|_\infty t^{q^*}}{q^*} = \varphi_{\epsilon,{\delta}}(t),
\end{equation*}
for all $\mu \geq 0$, by inequality~\eqref{3.2.6}.
So in view of \eqref{eq3.2.19}, it suffices to establish
\begin{equation*}\label{eqsufficient_condition}
\sup_{t \geq 0} \varphi_{\epsilon,{\delta}}(t) < \frac{m}{N} \frac{(a_0 S_q)^{\frac{N}{mq}}}{\|b\|_\infty^{\frac{N-mq}{mq}}},
\end{equation*}
for sufficiently small $\epsilon > 0$.
Suppose, on contrary,  there exist sequences $\epsilon_j \to 0$ and $t_j > 0$ such that
\begin{align}
\varphi_{\epsilon_j,{\delta_j}}(t_j) &= \frac{t_j^p}{p} \int_{\Omega} |\nabla^m v_j|^p \, \mathrm{d}x + \frac{a_0 t_j^q}{q} \int_{\Omega} |\nabla^m v_j|^q \, \mathrm{d}x - \frac{c_0 t_j^s}{s} \int_{\Omega} v_j^s \, \mathrm{d}x - \frac{\|b\|_\infty t_j^{q^*}}{q^*} \geq \frac{m}{N} \frac{(a_0 S_q)^{\frac{N}{mq}}}{\|b\|_\infty^{\frac{N-mq}{mq}}}, \label{eqcontradiction_assumption}\\
t_j \varphi'_{\epsilon_j,{\delta_j}}(t_j) &= t_j^p \int_{\Omega} |\nabla^m v_j|^p \, \mathrm{d}x + a_0 t_j^q \int_{\Omega} |\nabla^m v_j|^q \, \mathrm{d}x - c_0 t_j^s \int_{\Omega} v_j^s \, \mathrm{d}x - \|b\|_\infty t_j^{q^*} = 0, \label{eqcritical_point_condition}
\end{align}
where $v_j = v_{\epsilon_j,\delta_j}$ and $\delta_j = \delta(\epsilon_j)$.
By the established estimates~\eqref{eq3.5.28}, \eqref{eq3.5.29}, and~\eqref{eq3.5.30}, we know that
\begin{align*}
\int_{\Omega} |{\nabla^m v_j}|^q \, \mathrm{d}x \to S_q,~~
\int_{\Omega} |{\nabla^m v_j}|^p \, \mathrm{d}x \to 0,~~\int_{\Omega} {v_j}^s \, \mathrm{d}x \to 0.
\end{align*}
The inequality~\eqref{eqcontradiction_assumption} implies that the sequence $(t_j)$ is bounded so along a subsequence, $t_j \to t_0 > 0$. Passing the limit in~\eqref{eqcritical_point_condition}, this yields
\begin{equation}\label{eqlimiting_relation}
a_0 S_q t_0^q - \|b\|_\infty t_0^{q^*} = 0,
\end{equation}
which gives $t_0 = \left(\frac{a_0 S_q}{\|b\|_\infty}\right)^{\frac{1}{(q^*-q)}}$.
Combining equations~\eqref{eqcritical_point_condition} and~\eqref{eqlimiting_relation} with the asymptotic estimates \eqref{eq3.5.28}, we obtain
\begin{equation*}\label{eqresidual_equation}
t_j^p \int_{\Omega} |{\nabla^m v_j}|^p \, \mathrm{d}x + a_0 S_q (t_j^q - t_0^q) - c_0 t_j^s \int_{\Omega} {v_j}^s \, \mathrm{d}x - \|b\|_\infty (t_j^{q^*} - t_0^{q^*}) = O\left( \left(\frac{\epsilon_j}{\delta_j}\right)^{\frac{N-mq}{q-1}} \right).
\end{equation*}
Applying the mean value theorem to the terms $(t_j^q - t_0^q)$ and $(t_j^{q^*} - t_0^{q^*})$, we write
\begin{equation}\label{eqmean_value_application}
\left( q a_0 S_q \sigma_j^{q-1} - q^* \|b\|_\infty \tau_j^{q^*-1} \right) (t_j - t_0) = c_0 t_j^s \int_{\Omega} {v_j}^s \, \mathrm{d}x - t_j^p \int_{\Omega} |{\nabla^m v_j}|^p \, \mathrm{d}x + O\left( \left(\frac{\epsilon_j}{\delta_j}\right)^{\frac{N-mq}{q-1}} \right),
\end{equation}
where $\sigma_j$ and $\tau_j$ are intermediate values between $t_0$ and $t_j$.
As $t_j \to t_0$, we have $\sigma_j, \tau_j \to t_0$ which leads to
\begin{equation*}
q a_0 S_q \sigma_j^{q-1} - q^* \|b\|_\infty \tau_j^{q^*-1} \to q a_0 S_q t_0^{q-1} - q^* \|b\|_\infty t_0^{q^*-1} = -(q^* - q) \|b\|_\infty t_0^{q^*-1} < 0,
\end{equation*}
where the last equality follows from~\eqref{eqlimiting_relation}.
Since the left-hand side of~\eqref{eqmean_value_application} approaches to a negative limit and the right-hand side approaches zero due to~\eqref{eqlimit_condition_1} and~\eqref{eqlimit_condition_2}, we conclude that $t_j \leq t_0$ for sufficiently large $j$.
Finally, dividing \eqref{eqcritical_point_condition} by $q^*$, subtracting from~\eqref{eqcontradiction_assumption}, and using \eqref{eq3.5.28} and \eqref{eqlimiting_relation}, we obtain
\begin{equation*}
\left( \frac{1}{p} - \frac{1}{q^*} \right) t_j^p \int_{\Omega} |{\nabla^m v_j}|^p \, \mathrm{d}x + \frac{m}{N} a_0 S_q t_j^q - c_0 \left( \frac{1}{s} - \frac{1}{q^*} \right) t_j^s \int_{\Omega} {v_j}^s \, \mathrm{d}x \geq \frac{m}{N} a_0 S_q t_0^q + O\left( \left(\frac{\epsilon_j}{\delta_j}\right)^{\frac{N-mq}{q-1}} \right).
\end{equation*}
Given that $t_j \leq t_0$ for large $j$ and applying the limiting conditions~\eqref{eqlimit_condition_1} and~\eqref{eqlimit_condition_2}, we derive
\begin{equation*}
c_0 \left( \frac{1}{s} - \frac{1}{q^*} \right) t_0^s \leq 0.
\end{equation*}
This is a contradiction, since $c_0 > 0$, $t_0 > 0$, and $s < q^*$.
Therefore, our assumption was false, and the lemma holds.
\end{proof}
\begin{proof}[\bf{Proof of Theorem \ref{thm3.2.4}}]
To complete the proof of Theorem \ref{thm3.2.4} via Lemma \ref{lemcritical_energy_bound}, we find \(\delta = \delta(\epsilon) \in (0, 1]\) such that \(\epsilon / \delta \to 0\) and conditions \eqref{eqlimit_condition_1} and \eqref{eqlimit_condition_2} hold as \(\epsilon \to 0\).

\textbf{Case (a)}: \(1 < p < \frac{N(q-1)}{N-m}\) and \(\frac{N^2(q-1)}{(N-m)(N-mq)} < s < q^*\).
Take \(\delta = \epsilon^\kappa\) with \(\kappa \in [0, 1)\) to be determined. Since
\[
s > \frac{N^2(q-1)}{(N-m)(N-mq)} > \frac{N(q-1)}{N-mq},
\]
equations \eqref{eq3.5.29} and \eqref{eq3.5.30} yield
\begin{align}
\frac{\int_{\Omega} |\nabla^m v_{\epsilon,\delta}|^p \, \mathrm{d}x}{\int_{\Omega} v_{\epsilon,\delta}^s \, \mathrm{d}x} = O\left(\epsilon^{\frac{1}{q}\left[(N-mq)\left(s+\frac{p}{q-1}\right)-Nq\right]+\frac{\kappa}{q-1}[N(q-1)-(N-m)p]}\right) = O\left(\epsilon^{\frac{\kappa - \underline{\kappa}}{q-1}[N(q-1)-(N-m)p]}\right), \label{eqcase1_limit1}
\end{align}
where
\begin{equation*}\label{eqkappa_lower}
\underline{\kappa} = \frac{Nq(q-1)-(N-mq)(q-1)s-(N-mq)p}{[N(q-1)-(N-m)p]q},
\end{equation*}
and
\begin{align}
\frac{\left(\frac{\epsilon}{\delta}\right)^{\frac{N-mq}{q-1}}}{\int_{\Omega} v_{\epsilon,\delta}^s \, \mathrm{d}x} = O\left(\epsilon^{\frac{1}{q(q-1)}\left((N-mq)(q-1)s-(Nq-2N+mq)q\right)-\frac{\kappa}{q-1}(N-mq)}\right) = O\left(\epsilon^{\frac{\overline{\kappa} - \kappa}{q-1}(N-mq)}\right), \label{eqcase1_limit2}
\end{align}
where
\begin{equation*}\label{eqkappa_upper}
\overline{\kappa} = \frac{(N - mq)(q - 1)s - (Nq - 2N + mq)q}{(N - mq)q}.
\end{equation*}
We want to choose \(\kappa \in [0, 1)\) satisfying \(\kappa > \underline{\kappa}\) and \(\kappa < \overline{\kappa}\), which is possible if and only if \(\underline{\kappa} < \overline{\kappa}\), \(\underline{\kappa} < 1\), and \(\overline{\kappa} > 0\). These conditions are respectively equivalent to
\begin{align*}
s &> \frac{N^2(q - 1)}{(N - m)(N - mq)},  \\
s &> \frac{Np}{N - mq},  \\
s &> \frac{N^2(q - 1)}{(N - m)(N - mq)} - \frac{N - mq}{(N - m)(q - 1)}, 
\end{align*}
all of which are satisfied under our assumptions on \(p\) and \(s\).
With such a choice of \(\kappa\), both expressions in \eqref{eqcase1_limit1} and \eqref{eqcase1_limit2} decay to zero as \(\epsilon \to 0\), verifying conditions \eqref{eqlimit_condition_1} and \eqref{eqlimit_condition_2}.

\textbf{Case (b)}: \(\frac{N(q-1)}{N-m} \leq p < q\) and \(\frac{Np}{N-mq} < s < q^*\).
Take \(\delta = 1\). Since
\[
s > \frac{Np}{N-mq} \geq \frac{N^2(q-1)}{(N-m)(N-mq)} > \frac{N(q-1)}{N-mq},
\]
equations \eqref{eq3.5.29} and \eqref{eq3.5.30} yield
\begin{equation*}\label{eqcase2_limit1}
\frac{\int_{\Omega} |\nabla^m v_{\epsilon,\delta}|^p \, \mathrm{d}x}{\int_{\Omega} v_{\epsilon,\delta}^s \, \mathrm{d}x} = 
\begin{cases}
O\left(\epsilon^{\frac{(N-mq)s-Np}{q}}\right), & \text{if } p > \frac{N(q-1)}{N-m}, \\
O\left(\epsilon^{\frac{(N-mq)s-Np}{q}}|\log \epsilon|\right), & \text{if } p = \frac{N(q-1)}{N-m},
\end{cases}
\end{equation*}
and
\begin{equation*}\label{eqcase2_limit2}
\frac{\left(\frac{\epsilon}{\delta}\right)^{\frac{N-mq}{q-1}}}{\int_{\Omega} v_{\epsilon,\delta}^s \, \mathrm{d}x} = O\left(\epsilon^{\frac{1}{q(q-1)}\left((N-mq)(q-1)s - (Nq- 2N+mq)q\right)}\right).
\end{equation*}
Since \(s > \frac{Np}{N-mq}\), the limit in \eqref{eqlimit_condition_1} holds. The limit in \eqref{eqlimit_condition_2} holds because
\begin{equation*}\label{eqcase2_verification}
\frac{Np}{N-mq} \geq \frac{N^2(q-1)}{(N-m)(N-mq)} > \frac{(Nq-2N+mq)q}{(N-mq)(q-1)}.
\end{equation*}
Therefore, both limiting conditions \eqref{eqlimit_condition_1} and \eqref{eqlimit_condition_2} are satisfied in Case (b).

Therefore, we have constructed appropriate parameter relationships \(\delta = \delta(\epsilon)\) such that the asymptotic conditions of Lemma \ref{lemcritical_energy_bound} are satisfied. This completes the proof of Theorem \ref{thm3.2.4}. 
\end{proof}
\begin{section}{Proof of nonexistence results}
In this section, we will establish the Poho\v{z}aev identity and further using this identity we will prove the nonexistence results for the problem \eqref{3.1.1}. Let us recall that we considered
\begin{equation}\label{3.7.1}
\begin{cases}
\mathcal{L}^m_{p,q}(u) = f(x,u) ~&\text{in } \Omega,\\[6pt]
    \nabla^\gamma u\big|_{\partial\Omega}=0
    &\text{for all  }\gamma\le m-1,
\end{cases}
\end{equation}
where $\mathcal{L}^m_{p,q}$ represents a polyharmonic double phase operator defined by
\begin{equation}
\mathcal{L}^m_{p,q}(u) = 
\begin{cases}
-\nabla.(\Delta^{\frac{m-1}{2}}\{ |\nabla\Delta^{\frac{m-1}{2}}u|^{p-2} \nabla\Delta^{\frac{m-1}{2}} u + a(x) |\nabla\Delta^{\frac{m-1}{2}}u|^{q-2} \nabla\Delta^{\frac{m-1}{2}} u \}) & \text{if } m \text{ is odd}, \\[2ex]
\Delta^{\frac{m}{2}}\left( |\Delta^{\frac{m}{2}} u|^{p-2} \Delta^{\frac{m}{2}} u + a(x) |\Delta^{\frac{m}{2}} u|^{q-2} \Delta^{\frac{m}{2}} u \right) & \text{if } m \text{ is even}.
\end{cases}
\end{equation}
 Here, $m \in \mathbb{N}$, ${0 \leq a(\cdot) \in C^{1}(\Omega)}$ and $\Omega \subset \mathbb{R}^N$ with $N \geq 2$ is a bounded domain, and
$f(x,u) = c(x) |u|^{r-2}u + \mu |u|^{p^*-2}u + b(x)|u|^{q^*-2}u.$
\begin{proof}[\bf{Proof of Theorem \ref{thm3.3.1}}]
First we multiply $(x.\nabla u)$  to both sides of {\eqref{3.7.1}} and integrate to get
\begin{equation}
    (-1)^m\int_{\Omega}\nabla^m(|\nabla^m u|^{p-2}\nabla ^mu+a(x)|\nabla^m u|^{q-2}\nabla ^mu)(x.\nabla u)\,\mathrm{d}x=\int_{\Omega}f(x,u)(x.\nabla u)\,\mathrm{d}x.
\end{equation}
Now, we apply the integration by parts formula, and straightforward calculations, we get
\begin{equation}
\begin{aligned}\label{3.7.4}
    &\int_{\Omega}\left(\left(m-\frac{N}{p}\right)|\nabla^m u|^p+\left(m-\frac{N}{q}\right)a(x)|\nabla^m u|^q\right)\,\mathrm{d}x-\int_{\Omega}\frac{|\nabla^mu|^q}{q}(\nabla a.x)\,\mathrm{d}x\\
    & \quad -\int_{\partial\Omega}\left(\left(1-\frac{1}{p}\right)|\nabla^mu|^p+\left(1-\frac{1}{q}\right)a(x)|\nabla^mu|^q\right)(x.\nu)\,\mathrm{d}\sigma=-N\int_{\Omega}F(x,u)\,\mathrm{d}x.
\end{aligned}
\end{equation}
Hence, using that $u$ solves \eqref{3.7.1} weakly and using \eqref{3.7.4}, we obtain
\begin{equation}
\begin{aligned}\label{3.7.5}
  &  \int_{\Omega}\left(\frac{1}{p^*}|\nabla^m u|^p+\frac{1}{q^*}a(x)|\nabla^m u|^q\right)\,\mathrm{d}x+\frac{1}{Nq}\int_{\Omega}{|\nabla^mu|^q}(\nabla a.x)\,\mathrm{d}x\\
   & \quad  +\frac{1}{N}\int_{\partial\Omega}\left(\left(1-\frac{1}{p}\right)|\nabla^mu|^p+\left(1-\frac{1}{q}\right)a(x)|\nabla^mu|^q\right)(x.\nu)\,\mathrm{d}\sigma=\int_{\Omega}F(x,u)\,\mathrm{d}x.
\end{aligned}
\end{equation}
Now multiplying the equation \eqref{3.7.1} by $u/q^*$, integrating over $\Omega$ and using integrating by parts, we get
\begin{equation}\label{3.7.6}
    \frac{1}{q^*}\int_{\Omega}(|\nabla^mu|^p+a(x)|\nabla ^mu|^q)\,\mathrm{d}x=\frac{1}{q^*}\int_{\Omega}f(x,u)u\,\mathrm{d}x.
\end{equation}
Finally subtracting \eqref{3.7.6} from \eqref{3.7.5} and using the fact that $\frac{1}{p^*}-\frac{1}{q^*}=\frac{1}{p}-\frac{1}{q}$, we get 
\begin{equation}
\begin{aligned}\label{3.7.7}
& \left( \frac{1}{p} - \frac{1}{q} \right) 
\int_{\Omega} |\nabla^m u|^p \, \mathrm{d}x 
+ \frac{1}{Nq} \int_{\Omega} |\nabla^m u|^q (\nabla a \cdot x) \, \mathrm{d}x \\
& \quad + \frac{1}{N} \int_{\partial \Omega} \Bigl[
\Bigl( 1 - \frac{1}{p} \Bigr) |{\nabla ^mu}|^p
+ \Bigl( 1 - \frac{1}{q} \Bigr) a(x) |{\nabla^mu}|^q
\Bigr] (x \cdot \nu) \, d\sigma
= \int_{\Omega} \Bigl[ F(x,u) - \frac{1}{q^*} f(x,u)u \Bigr]\, \mathrm{d}x,
\end{aligned}
\end{equation}
which finishes the proof.
\end{proof}
\subsection{Proofs of Theorem \ref{thm3.3.2} and \ref{thm3.3.3}}
In this subsection, we consider the case where $\Omega$ is a star-shaped domain. Thus, without loss of generality, we assume that $0 \in \Omega$ and $\Omega$ is star-shaped with respect to the origin, which implies that $(x \cdot \nu) \geq 0$ on $\partial \Omega$. Since $a(x)$ is nonnegative, the boundary contribution in \eqref{3.7.7} is also nonnegative. Consequently, the Pohožaev-type identity \eqref{3.7.7} yields the inequality
\begin{equation}
\begin{aligned}\label{3.7.8}
 \left( \frac{1}{p} - \frac{1}{q} \right) 
\int_{\Omega} |\nabla^m u|^p \, \mathrm{d}x 
+ \frac{1}{Nq} \int_{\Omega} |\nabla^m u|^q (\nabla a \cdot x) \, \mathrm{d}x 
\leq \int_{\Omega} \Bigl[ F(x,u) - \frac{1}{q^*} f(x,u)u \Bigr] \,\mathrm{d}x.
\end{aligned}
\end{equation}
We are now in a position to prove the nonexistence theorem.
\begin{proof}[\bf{Proof of Theorem \ref{thm3.3.2}}]
Since $a(x)$ is radial and radially nondecreasing, it follows that $(\nabla a \cdot x) \geq 0$. Hence, by \eqref{3.7.8}, we obtain
\begin{equation}\label{3.7.9}
 \left( \frac{1}{p} - \frac{1}{q} \right) 
\int_{\Omega} |\nabla^m u|^p \, \mathrm{d}x
\leq \int_{\Omega} \Bigl[ c(x)\left(\frac{1}{r}-\frac{1}{q^*}\right)|u|^r+\mu\left(\frac{1}{p^*}-\frac{1}{q^*}\right)|u|^{p^*}\Bigr]\,\mathrm{d}x.
\end{equation}
From case (a), we deduce that
$$\left( \frac{1}{p} - \frac{1}{q} \right) 
\int_{\Omega} |\nabla^m u|^p \, \mathrm{d}x
\leq 0,$$
which yields $u \equiv 0$. By using case (b) and inequality \eqref{3.7.9}, we infer that
$$\left( \frac{1}{p} - \frac{1}{q} \right) 
\int_{\Omega} |\nabla^m u|^p \, \mathrm{d}x
\leq c_{\infty}\left(\frac{1}{p}-\frac{1}{q^*}\right)\int_{\Omega}|u|^p\,\mathrm{d}x.  $$
Hence, by the variational characterization of the first eigenvalue $\lambda_{1}(p)$ of the m-th order polyharmonic operator, we obtain
$$\left(\left( \frac{1}{p} - \frac{1}{q} \right)-\frac{ c_{\infty}}{\lambda_1(p)}\left(\frac{1}{p}-\frac{1}{q^*}\right)\right) 
\int_{\Omega} |\nabla^m u|^p \, \mathrm{d}x\leq0, $$
which, under the assumption on $c_{\infty}$, once again yields $u \equiv 0$.
\end{proof}

\begin{proof}[\bf{Proof of Theorem {\ref{thm3.3.3}}}]
For the sake of contradiction, assume that there exists a sequence $(u_j)_{j\in\mathbb{N}} \subset W_0^{m,\mathcal{H}}(\Omega) \cap W^{m+1,\mathcal{H}}(\Omega)$ of solutions to \eqref{3.7.1} with $\|u_j\| \to 0$ as $j \to \infty$. Let $\gamma > q^*$. Multiplying equation \eqref{3.7.1} by $u/\gamma$, we obtain, for $f$ as in the hypothesis,
$$\frac{1}{\gamma}\int_{\Omega}(\nabla^mu_j|^p+a(x)|\nabla^mu_j|^q)\,\mathrm{d}x=\frac{1}{\gamma}\int_{\Omega}(c(x)|u_j|^r+\mu|u_j|^{p^*}+b(x)|u_j|^{q^*})\,\mathrm{d}x.$$
Subtracting this identity from \eqref{3.7.5}, and using that $\Omega$ is star-shaped while $a$ is radial and radially nondecreasing, we deduce that
\begin{equation*}
    \begin{aligned}
    &\left(\frac{1}{p^*}-\frac{1}{\gamma}\right)\int_{\Omega}|\nabla^mu_j|^p\,\mathrm{d}x +\left(\frac{1}{q^*}-\frac{1}{\gamma}\right)\int_{\Omega}a(x)|\nabla^mu_j|^q\,\mathrm{d}x\\
    &\quad\leq \left(\frac{1}{r}-\frac{1}{\gamma}\right)\int_{\Omega}c(x)|u_j|^r\,\mathrm{d}x+\mu\left(\frac{1}{p^*}-\frac{1}{\gamma}\right)\int_{\Omega}|u_j|^{p^*}\,\mathrm{d}x+\left(\frac{1}{q^*}-\frac{1}{\gamma}\right)\int_{\Omega}b(x)|u_j|^{q^*}\,\mathrm{d}x\\
    &\quad\leq \left(\frac{1}{r}-\frac{1}{\gamma}\right)\int_{\Omega}c(x)|u_j|^r\,\mathrm{d}x+\mu\left(\frac{1}{p^*}-\frac{1}{\gamma}\right)\frac{1}{S_p^{\frac{p^*}{p}}}(\rho_{\mathcal{H}}(\nabla^mu_j))^{\frac{p^*}{p}}+\kappa\left(\frac{1}{q^*}-\frac{1}{\gamma}\right)(\rho_{\mathcal{H}}(\nabla^mu_j))^{\frac{q^*}{q}},
    \end{aligned}
\end{equation*}
where, in the final steps, we applied \eqref{eq3.3.6} and \eqref{eq3.3.7}. Hence, the resulting chain of inequalities yields
\begin{equation}\label{3.7.12}
    \begin{aligned}
        \left(\frac{1}{q^*}-\frac{1}{\gamma}\right)\rho_{\mathcal{H}}(\nabla^mu_j)&\leq\left(\frac{1}{r}-\frac{1}{\gamma}\right)\int_{\Omega}c(x)|u_j|^r\,\mathrm{d}x+\mu\left(\frac{1}{p^*}-\frac{1}{\gamma}\right)\frac{1}{S_p^{\frac{p^*}{p}}}(\rho_{\mathcal{H}}(\nabla^mu_j))^{\frac{p^*}{p}}\\
        &\quad+\kappa\left(\frac{1}{q^*}-\frac{1}{\gamma}\right)(\rho_{\mathcal{H}}(\nabla^mu_j))^{\frac{q^*}{q}},
    \end{aligned}
\end{equation}
Now, if $r \in (p, p^*]$, the Sobolev embedding $W_0^{m,p}(\Omega) \hookrightarrow L^r(\Omega)$ yields
\begin{equation*}
    \begin{aligned}
        \left(\frac{1}{q^*}-\frac{1}{\gamma}\right)\rho_{\mathcal{H}}(\nabla^mu_j)\leq&\left(\frac{1}{r}-\frac{1}{\gamma}\right)c_\infty C_s\left(\int_{\Omega}|\nabla^mu_j|^p\,\mathrm{d}x\right)^{r/p}+\mu\left(\frac{1}{p^*}-\frac{1}{\gamma}\right)\frac{1}{S_p^{\frac{p^*}{p}}}(\rho_{\mathcal{H}}(\nabla^mu_j))^{\frac{p^*}{p}}\\
        &\quad
        +\kappa\left(\frac{1}{q^*}-\frac{1}{\gamma}\right)(\rho_{\mathcal{H}}(\nabla^mu_j))^{\frac{q^*}{q}}\\
        &\leq \left(\frac{1}{r}-\frac{1}{\gamma}\right)c_\infty C_s\left(\rho_{\mathcal{H}}(\nabla^mu_j)\right)^{r/p}+\mu\left(\frac{1}{p^*}-\frac{1}{\gamma}\right)\frac{1}{S_p^{\frac{p^*}{p}}}(\rho_{\mathcal{H}}(\nabla^mu_j))^{\frac{p^*}{p}}\\
        &\quad+\kappa\left(\frac{1}{q^*}-\frac{1}{\gamma}\right)(\rho_{\mathcal{H}}(\nabla^mu_j))^{\frac{q^*}{q}}.
    \end{aligned}
\end{equation*}
Since, by \eqref{eq3.3.4}, we also have $\rho_{\mathcal{H}}(\nabla^{m}u_j)\to 0$, this yields a contradiction, because all the exponents of $\rho_{\mathcal{H}}(\nabla^{m}u_j)$ on the right-hand side are strictly larger than $1$. Hence, case (a) is proved.

For case (b), set $\mathcal{C}(x,t):= t^{\sigma}+c(x)\,t^{r}$ with $\sigma\in(q,p^*)$. By Proposition \ref{prop3.3.12}, we have the embedding $W^{m,\mathcal{H}}(\Omega)\hookrightarrow L^{\mathcal{C}}(\Omega)$. Thus,
$$\int_{\Omega}c(x)|u_j|^r\,\mathrm{d}x\leq\rho_{\mathcal{C}}(u_j)\leq\max\{\|u_j\|_{\mathcal{C}}^\sigma,\|u_j\|_{\mathcal{C}}^r\}\leq C_s\max\{\|u_j\|^\sigma,\|u_j\|^r\}=C_s\|u_j\|^{\sigma},$$
where we have used that $\|u_j\|<1$ for $j$ large enough. Therefore, combining with \eqref{3.7.12} and \eqref{eq3.3.4}, we have  
\begin{equation*}
    \begin{aligned}
        \left(\frac{1}{q^*}-\frac{1}{\gamma}\right)\|u_j\|^q\leq&\left(\frac{1}{q^*}-\frac{1}{\gamma}\right)\rho_{\mathcal{H}}(\nabla^mu_j)\leq\left(\frac{1}{r}-\frac{1}{\gamma}\right)C_s\|u_j\|^\sigma\\
        &\quad +\mu\left(\frac{1}{p^*}-\frac{1}{\gamma}\right)\frac{1}{S_p^{\frac{p^*}{p}}}\|u_j\|^{p^*}+\kappa\left(\frac{1}{q^*}-\frac{1}{\gamma}\right)\|u_j\|^{\frac{pq^*}{q}},
    \end{aligned}
\end{equation*}
which contradicts the fact that $\|u_j\| \to 0$, since all the exponents of $\|u_j\|$ on the right-hand side exceed $q$ in particular, $\frac{pq^*}{q} > q$, which follows from the condition $q/p \leq N/(N-1)$. This completes the proof of case (b).
Finally, to establish case (c), we first estimate
\begin{equation*}
\begin{aligned}
\int_{\Omega} c(x) |u|^r\, \mathrm{d}x &\leq c_\infty
\int_{\operatorname{supp}(c)} |u|^r\, \mathrm{d}x \leq c_\infty C_s
\left(
\int_{\operatorname{supp}(c)} |\nabla^m u|^q \,\mathrm{d}x
\right)^{r/q}\\
&\leq \frac{c_\infty C_s}{
(a'_0)^{r/q}}
\left(
\int_{\operatorname{supp}(c)} a(x) |\nabla^m u|^q \,\mathrm{d}x
\right)^{r/q}
\leq \frac{c_\infty C_s}{
(a'_0)^{r/q}} (\rho_{\mathcal{H}}(\nabla^m u))^{r/q},
\end{aligned}
\end{equation*}
thanks to the embedding $W^{m,q}(\text{supp}(c)) \hookrightarrow L^r(\text{supp}(c))$. Then the conclusion follows exactly as in case (a).
\end{proof}
 \end{section}
\section*{Acknowledgements}
Ashutosh Dixit sincerely thanks the Government of India for the DST INSPIRE Fellowship (Ref. No.  DST/INSPIRE Fellowship/2022/IF220580) and gratefully acknowledges their support.

\section*{Data Availability}
 No data were generated or analyzed in this study.
\section*{Competing Interest statement}
No potential conflict of interest was reported by the authors.

\end{document}